\documentclass[preprint,12pt]{amsart}
\usepackage[utf8]{inputenc}
\usepackage[margin=.8in]{geometry}  
\newcommand{\miniplus }{\!+\!}
\usepackage{verbatim,mathtools}
\usepackage{graphicx,wrapfig}              
\usepackage{amsmath,amsfonts,amsthm,amssymb}
\usepackage{pifont,stmaryrd}
\usepackage{tikz} 
\usepackage{relsize}
\usetikzlibrary{shapes}
\usetikzlibrary{positioning}
\usepackage{ytableau,enumitem}   
\usepackage{xcolor}
\usepackage{float}
\usepackage{pifont}
\usepackage{multirow}

\newtheorem{thm}{Theorem}[section]
\newtheorem{lemma}[thm]{Lemma}

\newtheorem{prop}[thm]{Proposition}
\newtheorem{corollary}[thm]{Corollary}

\definecolor{ultragreen}{RGB}{24,120,50} \theoremstyle{remark}
\newtheorem{example}[thm]{\color{ultragreen}Example}
\theoremstyle{plain}
\newtheorem{conj}[thm]{Conjecture}
\newtheorem{defn}[thm]{Definition}
\newtheorem{rmk}[thm]{Remark}
\newtheorem{soru}[thm]{Question}
\newcommand{\diplus}{\searrow}
\newcommand{\revdiplus}{\nearrow}

\newcommand{\rank}{\mathfrak{R}}

\numberwithin{thm}{section}
\newtheoremstyle{TheoremNum}
{\topsep}{\topsep}              
{\itshape}                      
{}                              
{\bfseries}                     
{.}                             
{ }                             
{\thmname{#1}\thmnote{ \bfseries #3}}
\theoremstyle{TheoremNum}
\newtheorem{thmn}{Theorem}

\newcommand{\ZZ}{\mathbb{Z}}      
\newcommand{\close}{\circlearrowright\! }
\newcommand{\revclose}{\circlearrowleft\! }

\newcommand{\clmat}{\overline{\mathfrak{M}}}

\newcommand{\rmm}{{\triangledown} \rank}
\newcommand{\drm}{{\rotatebox[origin=c]{180}{$\triangledown$}} \rank}
\newcommand{\dualfix}[1]{{#1}^*}
\DeclareTextFontCommand{\definition}{\color{blue}\em}
\usetikzlibrary{patterns,decorations.pathreplacing} 
\begin{document}



\title{Oriented posets, Rank Matrices and q-deformed Markov Numbers}
\author{EZGİ KANTARCI OĞUZ}
\address{Dept. of Mathematics, 
Galatasaray  University, 
İstanbul} 
\email{ezgikantarcioguz@gmail.com}
\thanks{The author was supported by Tübitak BİDEB grant 121C285.}

\begin{abstract}
		We define \emph{oriented posets} with corresponding \emph{rank matrices}, where linking two posets by an edge corresponds to matrix multiplication. In particular, linking chains via this method gives us fence posets, and taking traces gives us circular fence posets. As an application, we give a combinatorial model for $q$-deformed Markov numbers.  We also resolve a conjecture of Leclere and Morier-Genoud and give several identities between circular rank polynomials.
\end{abstract}

\maketitle

	
	
\section{Introduction}

 In this paper, we introduce a new tool motivated by the recent work on $q$-deformations of rationals (\cite{originalconj}, \cite{leclere}), fence posets and their rank polynomials (\cite{Saganpaper}, \cite{ourpaper}). \definition{Oriented posets} are posets with specialized end points that can be connected together to build larger posets. We can keep track of this building operation by matrix multiplication, which allows for fast and easy calculations. In this paper, we develop their combinatorics and consider applications on rank polynomial identities and Markov numbers. Oriented posets so far seem to be versatile objects that are applicable to other branches of combinatorics. See \cite{cluster} and \cite{chainlink} for applications in calculating cluster algebra coefficients and studying polytopes respectively.

Let us start with our motivating example of fence posets. Fence posets are combinatorial objects that come up in a variety of settings, including cluster algebras, quiver representations, snake graphs and continued fractions (See \cite{MR3908893}, \cite{MR4158076}). Consider composition $\alpha=(\alpha_1,\alpha_2,\ldots,\alpha_s)$ of $n$, where the number $s$ of parts is called the \definition{length} of $\alpha$, denoted $\ell(\alpha)$ and $n$ is called the size of $\alpha$, denoted $|\alpha|$. We define an $n+1$ element poset $F(\alpha)$, called the \definition{fence poset} of $\alpha$ by taking the transitive closure of the following cover relations: 
	\begin{equation*}
		x_1\preceq x_2 \preceq \cdots\preceq x_{\alpha_1+1}\succeq x_{\alpha_1+2}\succeq \cdots\preceq x_{\alpha_1+\alpha_2+1}\preceq x_{\alpha_1+\alpha_2+2}\preceq\cdots\preceq x_{\alpha_1+\alpha_2+\alpha_3+1}\succeq \cdots
	\end{equation*}
	The maximal chain corresponding to each $\alpha_i$ is called a  \definition{segment} of the poset. If the length of $\alpha$ is even, we also define the \definition{circular fence poset} $\overline{F}(\alpha)$ corresonding to $\alpha$ by setting $x_{n+1}=x_1$ in $F(\alpha)$. Note that unlike the regular fence poset, the number of nodes in $\overline{F}(\alpha)$ matches the size of $\alpha$.
	
	\begin{example} For $\alpha=(2,1,1,3)$ the regular and circular fence posets are depicted in Figure~$\ref{fig:2113}$ below. \label{ex:rankpolyfor2113}
		
		\begin{figure}[ht]\centering
			\begin{tikzpicture}[scale=.7]
				\fill(0,1) circle(.1);
				\fill(1,2) circle(.1);
				\fill(2,3) circle(.1);
				\fill(3,2) circle(.1);
				\fill(4,3) circle(.1);
				\fill(5,7/3) circle(.1);
				\fill(6,5/3) circle(.1);
				\fill(7,1) circle(.1);
				\draw (0,1)--(2,3)--(3,2)--(4,3)--(7,1);
				\draw (0,.5) node{$x_1$};
				\draw (1,1.5) node{$x_2$};
				\draw (2,2.5) node{$x_3$};
				\draw (3,1.5) node{$x_4$};
				\draw (4,2.5) node{$x_5$};
				\draw (5,.5+4/3) node{$x_6$};
				\draw (6,.5+2/3) node{$x_7$};
				\draw (7,.5) node{$x_8$};
			\end{tikzpicture}\qquad
			\begin{tikzpicture}[scale=.7]
				\fill(1,1) circle(.1);
				\fill(2,2) circle(.1);
				\fill(3,1) circle(.1);
				\fill(4,2) circle(.1);
				\fill(5.25,2-2/3) circle(.1);
				\fill(6.5,2-4/3) circle(.1);
				\fill(4,0) circle(.1);
				\draw (4,0)--(1,1)--(2,2)--(3,1)--(4,2)--(6.5,2-4/3)--(4,0);
				\draw (4,-.5) node{$x_1$};
				\draw (1,.5) node{$x_2$};
				\draw (2,1.5) node{$x_3$};
				\draw (3.3,.7) node{$x_4$};
				\draw (4,1.5) node{$x_5$};
				\draw (5.25,1.5-2/3) node{$x_6$};
				\draw (6.5,1.5-4/3) node{$x_7$};
			\end{tikzpicture}
			\caption{Posets $F(2,1,1,3)$ (left) and $\overline{F}(2,1,1,3)$ (right).}\label{fig:2113} 
		\end{figure}
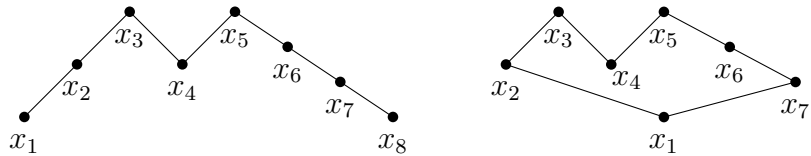
		
	\end{example}

	Down-closed subsets (subsets $I$ which satisfy $x\in I$, $y\preceq x \Rightarrow y \in I$) of a fence are called its \definition{ideals}. Ideals of a fence poset $F(\alpha)$ (resp. circular fence poset $\overline{F}(\alpha)$) ordered by inclusion have the structure of a distributive lattice ranked by the number of elements in each ideal. We denote this lattice by $J(\alpha)$ (resp. $\overline{J}(\alpha)$). The generating polynomials of $J(\alpha)$ and $\overline{J}(\alpha)$ are called the \definition{rank polynomial} of $\alpha$ and the \definition{circular rank polynomial} of $\alpha$ respectively:
	\begin{align*}
		\rank(\alpha;q)=& \sum_{I \in J(\alpha)} q^{|I|},\qquad \qquad \overline{\rank}(\alpha;q)= \sum_{I\in \overline{J}(\alpha)} q^{|I|}.
	\end{align*}
	
	\begin{example} \label{ex:2113} For the regular and circular fence posets for $\alpha=(2,1,1,3)$ illustrated in Figure \ref{fig:2113}, we get the following rank polynomials:
		\begin{align*}
			\rank((2,1,1,3);q)=&1+3q+5q^2+6q^3+6q^4+5q^5+3q^6+2q^7+q^8,\\
			\overline{\rank}((2,1,1,3);q)=& 1+2q+3q^2+4q^3+4q^4+3q^5+2q^6+q^7=[5]_q[4]_q.
		\end{align*}
	\end{example}

Rank polynomials recently came up in Ovsienko and Morier-Genoud's work \cite{originalconj} defining a $q$-deformation of rational numbers  using their continued fraction expansions. For a given rational number $r/t$, the $q$-deformation $[r/t]_q$ equals $R(q)/T(q)$ where $R(q), T(q) \in \mathbb{Z}[q]$ are polynomials that,  at $q=1$,  evaluate to $r$ and $t$ respectively. When $\frac{r}{t}\geq 1$, one can find a corresponding composition $\alpha_{r/t}$ such that $R(q)$ equals the rank polynomial of $\rank(\alpha_{r/t};q)$ and $T(q)$ equals the rank polynomial of the same fence poset with leftmost segment deleted. The exact nature of these constructions is described in Section \ref{sec:prelim} below. Ovsienko and Morier-Genoud's observations on this front led them to conjecture that the rank polynomials of fence posets are unimodal. This conjecture was extended to describe the interlacing properties in \cite{Saganpaper}, and the extended version was proved in \cite{ourpaper} using circular fence posets. It was further shown that:
	\begin{thm}[\cite{ourpaper}]
		Circular rank polynomials $\overline{\rank}((\alpha_1,\alpha_2,\ldots,\alpha_{2s});q)$ are symmetric with non-negative integer coefficients that remain constant under cyclic shifts of parts of $\alpha$.
		\begin{align}
			\overline{\rank}((\alpha_1,\alpha_2,\ldots,\alpha_{2s});q)=\overline{\rank}((\alpha_{2s},\alpha_1,\alpha_2,\ldots,\alpha_{2s-1});q). \label{eq:cyclic shift}
		\end{align} The polynomials are also unimodal if one of the following is satisfied:\label{thm:ourtheorem}
		\begin{itemize}[noitemsep]
			\item $\alpha_1+\alpha_2+\cdots+\alpha_{2s}$ is odd,
			\item $\alpha$ has two consecutive parts that are larger than $1$,
			\item $\alpha$ has three consecutive parts $k,1,l$ with $|k-l|>1$.
		\end{itemize} 
	\end{thm}
It was in fact conjectured in \cite{ourpaper} that the circular rank polynomial is unimodal if the number of segments is different than $4$.

\begin{conj}[\cite{ourpaper}] \label{ourconj} $\overline{\rank}(\alpha;q)$ is unimodal except for the cases $\alpha=(1,k,1,k)$ or $(k,1,k,1)$ for some positive integer $k$.
\end{conj}	
	
	In \cite{leclere}, Leclere and Morier-Genoud work with a $q$-deformation of $\text{PSL}(2,\ZZ)$. For a $r/t$ with corresponding regular and negative continued fraction expressions  $[a_1,a_2,\ldots,a_{2m}]$ and $\llbracket  c_1,c_2,\ldots,c_k\rrbracket$ respectively, they work with matrices $M_q(c_1,c_2,\ldots,c_k)$ and 	$M_q^+(a_1,a_2,\ldots,a_{2m})$ in $\text{PSL}_q(2,\ZZ)$, previously defined in \cite{originalconj}. The top and bottom left entries of $M_q(c_1,c_2,\ldots,c_k)$ are given by $R(q)$ and $T(q)$ whereas the top and bottom left entries of $M_q^+(a_1,a_2,\ldots,a_{2m})$ are given by $q R(q)$ and $q T(q)$ for some $z\in\mathbb{Z}$. 
	
	The traces of $M_q(c_1,c_2,\ldots,c_k)$, up to a multiple $\pm q^N$, are shown to be symmetric polynomials on $q$ with non-negative integer entries. Also, in a similar flavor to the circular rank polynomials, they are invariant under cyclic shifts of the sequence $\llbracket c_1,c_2,\ldots,c_k\rrbracket$. Leclere and Morier-Genoud further conjecture the following:
	\begin{conj}[\cite{leclere}, Conjecture 3.12]\label{conj:leclere} For any sequence $c_1,c_2,\ldots,c_k$ of integers, the trace $\mathrm{tr}(M_q(c_1,c_2,\ldots,c_k))$ is a polynomial in $\ZZ[q,q^{-1}]$ with unimodal coefficients.
	\end{conj}
	
	The matrices $M_q$ and $M^+_q$ as rank matrices of oriented posets, which leads us to conclude the $\mathrm{tr}(M(c_1,c_2,\ldots,c_k))$ is indeed a circular rank polynomial. In fact, one can show that (see Proposition \ref{T:1}) for $r/t\geq 1$ with regular and negative continued fraction expressions given by  $[a_1,a_2,\cdots,a_{2m}]$ and $\llbracket c_1,c_2,\ldots,c_k\rrbracket$ respectively, we have:
		\begin{align*}
			\mathrm{tr}(M_q(c_1,c_2,\ldots,c_k)&=\overline{\rank}((a_1-1,a_2,\ldots,a_{2m});q),\\
			\mathrm{tr}(M_q^+(a_1,a_2,\ldots,a_{2m}))&=\overline{\rank}((a_1,a_2,\ldots,a_{2m});q).
		\end{align*}
	This implies in particular that Conjecture \ref{conj:leclere} does not hold, but one can apply the results from Theorem \ref{thm:ourtheorem} to characterize large families where we have unimodality.
	
	In Section \ref{sec:oriented posets}, we define \emph{oriented posets} $\overrightarrow{P}:=(P,x_L,x_R)$ with specialized start and end points. Instead of rank polynomials, these posets come with $2\times 2$ \emph{rank matrices} $\rmm_q(\overrightarrow{P})$, where multiplication corresponds to combining two posets, and we can connect the two endpoints of a poset by simply taking the trace. In particular, any fence poset can be built from copies of the poset $\overrightarrow{\bullet}$ consisting of a single node and made circular by taking the trace (See Section \ref{sec:fence}). This allows us prove Proposition \ref{T:1} as well as construct new identities between infinite families of rank polynomials in Section \ref{sec:identities}. 
	
	\begin{thmn}[\ref{thm:3}] Let $X$ be a palindromic composition with an even number of parts. For $k\geq 1$, $r\geq1$ we have:
		\begin{align}
	\overline{\rank}((1,k,r+1,X,r);q)&= [k+1]_q \cdot \overline{\rank}((r+2,X,r);q), \tag{Id 1}\\
	\overline{\rank}((k,1,k+r,X,r);q)&=[k+1]_q \cdot
	\overline{\rank}((k+r+1,X,r);q)\tag{Id 2}.
\end{align}
	\end{thmn}

As an application, we give a combinatorial model for the $q$-deformed Markov numbers in terms of fence posets.
Solutions of the Markov equation $x^2+y^2+z^2=3xyz$ are called \definition{Markov triples}. Defined by A. Markoff in his thesis in 1880~\cite{thesis}, they are motivated by the study of rational approximations and Diophantine analysis. All Markov triples can be calculated from the initial solution $(1,1,1)$ via operations of the form $(x_0,y_0,z_0) \rightarrow (x_0,y_0, (x_0^2+y_0^2)/z_0)$. The solutions form an infinite binary tree called the \emph{Markov tree}. Some equivalent constructions are Christoffel words under the operation $(u,uv,v)\rightarrow (u,uuv,uv)$, Cohn matrices under matrix multiplication, simple fractions in the interval $[0,1]$ under the Farey sum, and mutations of the quiver representation for a torus with one marked point (\cite{markofftree}, \cite{chris}, \cite{MR4103773}).  Frobenius conjectured in $1912$ that each Markov number appears as the maximum of a unique Markov triple, a conjecture still open to this day (The interested reader is refered to  the 2012 book \cite{aignerbook} on the subject).
Markov numbers were recently $q$-deformed using Cohn matrices (  \cite{leclere}, \cite{kogiso}). In Section \ref{sec:cohn}, we use oriented posets to give a combinatorial model to directly calculate $q$-deformed Markov numbers.

\begin{thmn}[\ref{thm:markov}] The $q$-deformed Markov number corresponding to the Christoffel $ab$-word $w(a,b)=a\hat{w}(a,b)b$ is the rank polynomial of the circular fence poset $\overline{F}(3,1,\alpha(\hat{w}))$, where $\alpha(\hat{w})$ is formed by replacing each $a$ by the pair $1,1$ and each $b$ by the pair $2,2$ in $\hat{w}$.
\end{thmn}

We finish with listing some future research directions and references to some recent work using oriented posets including a resolution of Conjecture~\ref{ourconj} (\cite{chainlink}).
	
	\section{Preliminaries and Notation}\label{sec:prelim}
	
	For visual clarity, we at times use the notation $\cdot$ for both regular and matrix multiplication. 
	
	\subsection{Fence Posets and Rank Polynomials}
	
	A composition $\alpha=(\alpha_1,\alpha_2,\ldots,\alpha_s)$ of $n$ is a list of non-negative integers that add up to $n$. Here $n$ is called the \definition{size} of $\alpha$, denoted $|\alpha|$, and $s$ is called the \definition{length} of $\alpha$. During this work, there are multiple instances when we add extra parts to a composition from the left. In that case, we slightly abuse notation and use  $(i,\alpha)$ to denote the composition $(i,\alpha_1,\alpha_2,\ldots,\alpha_s)$. We also use the notation $a^b$ to denote $b$ consecutive parts of size $a$.
	
	We follow the convention of \cite{Saganpaper} to start our fence posets with an up step. The rank polynomials for the posets starting with down steps can easily be inferred from these for the following reason: The complement of any ideal is an up-closed subset of the poset (satisfies $x\in I$, $y\succeq x \Rightarrow y \in I$). If we flip a fence poset $F(\alpha)$ upside down, the resulting poset has the rank polynomial given by reversing the coefficients of $\alpha$:
	\begin{align*}
		\rank^\updownarrow(\alpha;q)&=q^{|\alpha+1|}\cdot \rank(\alpha;q^{-1}).
	\end{align*}
When the fence we are considering is circular, turning the poset upside down is the same thing as cyclicly shifting segments by one, so the symmetry of the rank polynomial is actually equivalent to the invariance under cyclic shifts.
	\begin{align*}
	\overline{\rank}^\updownarrow(\alpha;q)&=\overline{\rank}((\alpha_{2s},\alpha_1,\alpha_2,\ldots,\alpha_{2s-1});q)=q^{|\alpha|}\cdot \overline{\rank}(\alpha,q^{-1})=\overline{\rank}(\alpha;q).
\end{align*}
	\begin{figure}[ht]\centering
	\begin{tikzpicture}[scale=.7]
		\fill(0,-1) circle(.1);
		\fill(1,-2) circle(.1);
		\fill(2,-3) circle(.1);
		\fill(3,-2) circle(.1);
		\fill(4,-3) circle(.1);
		\fill(5,-7/3) circle(.1);
		\fill(6,-5/3) circle(.1);
		\fill(7,-1) circle(.1);
		\draw (0,-1)--(2,-3)--(3,-2)--(4,-3)--(7,-1);
		\draw (0,-1.5) node{$x_1$};
		\draw (1,-2.5) node{$x_2$};
		\draw (2,-3.5) node{$x_3$};
		\draw (3,-2.5) node{$x_4$};
		\draw (4,-3.5) node{$x_5$};
		\draw (5,-1.5-4/3) node{$x_6$};
		\draw (6,-1.5-2/3) node{$x_7$};
		\draw (7,-1.5) node{$x_8$};
	\end{tikzpicture}\qquad
	\begin{tikzpicture}[scale=.7]
		\fill(1,-1) circle(.1);
		\fill(2,-2) circle(.1);
		\fill(3,-1) circle(.1);
		\fill(4,-2) circle(.1);
		\fill(5.25,-2+2/3) circle(.1);
		\fill(6.5,-2+4/3) circle(.1);
		\fill(4,0) circle(.1);
		\draw (4,0)--(1,-1)--(2,-2)--(3,-1)--(4,-2)--(6.5,-2+4/3)--(4,0);
		\draw (4,-.5) node{$x_1$};
		\draw (1,-1.5) node{$x_2$};
		\draw (2,-2.5) node{$x_3$};
		\draw (3.3,-1.7) node{$x_4$};
		\draw (4,-2.5) node{$x_5$};
		\draw (5.25,-2.5+2/3) node{$x_6$};
		\draw (6.5,-2.5+4/3) node{$x_7$};
	\end{tikzpicture}
	\caption{Posets $F^\updownarrow(2,1,1,3)$ (left) and $\overline{F}^\updownarrow(2,1,1,3)=\overline{F}(1,1,3,2)$ (right).}\label{fig:2113rev} 
\end{figure}

	\begin{example} The upside down reflection of the regular and circular fence posets for $\alpha=(2,1,1,3)$ are illustrated in Figure \ref{fig:2113rev}. The corresponding rank polynomials are:
	\begin{align*}
		\rank^\updownarrow((2,1,1,3);q)&=q^8 \cdot \rank((2,1,1,3);q^{-1})=1+2q+3q^2+5q^3+6q^4+6q^5+5q^6+3q^7+q^8,\\
			\overline{\rank}^\updownarrow((2,1,1,3);q)&=\overline{\rank}((2,1,1,3);q)= 1+2q+3q^2+4q^3+4q^4+3q^5+2q^6+q^7=[5]_q[4]_q.
	\end{align*}
\end{example}
	\subsection{$q$-Rationals}

	For a given rational number $r/t$, two natural ways to express  $r/t$ as a continued fraction are:
	\begin{align}
		\displaystyle
		\frac{r}{t}=a_1+\cfrac{1}{a_2+\cfrac{1}{a_3+\cfrac{1}{\ddots+\cfrac{1}{a_{2m}}}}}=c_1-\cfrac{1}{c_2-\cfrac{1}{c_3-\cfrac{1}{\ddots-\cfrac{1}{c_k}}}}\label{eq:twoqfractions}
	\end{align}
	
	The expressions $[a_1,a_2,\ldots,a_{2m}]$ and $\llbracket c_1,c_2,\ldots,c_k\rrbracket$ are called the regular and negative continued fractions for $r/t$ respectively. They are unique under the conditions $a_i\in \mathbb{Z}$, $a_i\geq 1$ for $i\geq 2$ for the regular case and $c_i\in \mathbb{Z}$, $c_i\geq 2$ for $i\geq 2$ for the negative case. Furthermore, one can easily convert one expression into the other via: 
	\begin{align}\label{eq:atocconversion}
		\llbracket c_1,c_2,\ldots,c_k\rrbracket=\llbracket a_1+1,2^{a_2-1},a_3+2,2^{a_4-1},a_5+2,\ldots,a_{2m-1}+2,2^{a_{2m}-1}\rrbracket
	\end{align} where $2^t$ stands for $t$ parts of size $2$. 
	The $q$-analogue for $\frac{r}{t}$ is then obtained by replacing the terms with powers of $q$ and $q$-integers in the following way:
	
	\begin{align}\displaystyle \label{eq:doubleq}
		\left [\frac{r}{t}\right]_q:=[a_1]_q+\cfrac{q^{a_1}}{[a_2]_{q^{-1}}+\cfrac{q^{-a_2}}{\ddots+\cfrac{q^{a_{2m-1}}}{[a_{2m}]_{q^{-1}}}}}=[c_1]_q-\cfrac{q^{c_1-1}}{[c_2]_q-\cfrac{q^{c_2-1}}{\ddots-\cfrac{q^{c_{k-1}-1}}{[c_k]_q}}}
	\end{align}
	
	Here, for a negative integer $-n$, we take $[-n]_q=-q^{-1}-q^{-2}-\cdots-q^{-n}$.
	
	The non-trivial fact that the two $q$-expressions given in Equation \ref{eq:twoqfractions} give the same $q$-analogue is shown in \cite{originalconj}. It is further shown that the said $q$-analogue equals a rational function $\frac{R(q)}{T(q)}$ with polynomials $R(q), T(q)$ evaluating to $r$ and $s$ respectively at $q=1$. 
	
	\begin{example}\label{ex:32/9} The fraction $32/7$ has the regular and negative continued fraction expressions $[3,1,1,4]$ and $[4,3,2,2,2]$.
\[\displaystyle\frac{32}{9}= 3+\cfrac{1}{1+\cfrac{1}{1+\cfrac{1}{4}}}=4-\cfrac{1}{3-\cfrac{1}{2-\cfrac{1}{2-\cfrac{1}{2}}}}.
\]
	We can calculate its $q$-analogue via Equation~\ref{eq:doubleq}.
		\begin{align*}\displaystyle\left [\frac{32}{9}\right]_q&= [3]_q+\cfrac{q^3}{[1]_{q^{-1}}+\cfrac{q^{-1}}{[1]_q+\cfrac{q}{[4]_{q^{-1}}}}}=[4]_q-\cfrac{q^4}{[3]_q-\cfrac{q^3}{[2]_q-\cfrac{q^2}{[2]_q-\cfrac{q^2}{[2]_q}}}}\\
			&=\frac{1 + 3q + 5q^2 + 6q^3 + 6q^4 + 5q^5 + 3q^6 + 2q^7 +  q^8}{1 + 2q + 2q^2 + 2q^3 +  q^4 +  q^5}=\frac{\rank((2,1,1,3);q)}{\rank((1,3);q)}.
		\end{align*}
	
	\end{example}
	
	In the case $\frac{r}{t}\geq 1$, the coefficient sequences satisfy $a_1\geq 1$ and $c_1\geq 2$, and it is possible to describe $R(q)$ and $T(q)$ in terms of rank polynomials as in Example~\ref{ex:32/9} above (The other cases can also be reduced to this case using the identity $[x+1]_q=q[x]_q+1$ from \cite{originalconj}). 
	
	For the rest of this work, we assume $\frac{r}{t}\geq 1$ with continued fraction expressions $[a_1,a_2,\ldots,a_{2m}]$ and $\llbracket c_1,c_2,\ldots,c_k\rrbracket$. In \cite{originalconj}, it is shown that one has: 
	\begin{align*}
		R(q)&=\rank((a_1-1,a_2,a_3,\ldots,a_{2m}-1);q),\\
		T(q)&=\rank((0,a_2-1,a_3,\ldots,a_{2m}-1);q).
	\end{align*} where if the first non-zero part of the composition is the second part, we take the fence poset that starts with a down step instead of an up step.
	
	This means, if we set $\alpha_{r/t}:=(a_1-1,a_2,a_3,\ldots,a_{2m}-1)$, $R(q)$ is given by the rank polynomial of $F(\alpha_{r/t})$ and $T(q)$ is the rank polynomial of the same poset with the leftmost segment deleted. Conversely, for a composition $\alpha=(u_1,d_1,u_2,d_2,\ldots,u_s,d_s)$ (here we allow $d_s=0$ for full generality), $\alpha=\alpha_{r/t}$ for the rational number $r/t$ with the following continued fraction expressions:
	\begin{align}\label{eq:alphatoatocconverion}
	    [u_1+1,d_1,u_2,d_2,\ldots,u_s,d_s+1]= \llbracket u_1+2,2^{d_1-1},u_2+2,2^{d_2-1},\ldots,u_s+1,2^{d_s}\rrbracket.
	\end{align}
	Note that if we further allow $u_1=0$, corresponding to the case of the fence posets starting with a down step, we can recover all possible $r/t\geq 1$ in this manner.
	
	\subsection{ $\text{PSL}_q(2,\ZZ)$} 
	In \cite{leclere}, Leclere and Morier-Genoud work with the $q$-deformation of $\text{PSL}(2,\ZZ)$ defined by the following $q$-versions of the standard generators considered modulo $\pm I$:
	
	\begin{align*}R_q=\begin{bmatrix} q&1\\0&1
		\end{bmatrix}, \qquad S_q=\begin{bmatrix} 0&-q^{-1}\\1&0 \end{bmatrix}.\end{align*}
	
	They set
	\begin{align}
		M_q(c_1,c_2,\ldots,c_k)&:=R_q^{c_1}S_q R_q^{c_2}S_q \cdots R_q^{c_k} S_q 
		,\label{eq:pic}\\
		M^+_q(a_1,a_2,\ldots,a_{2m})&:=R_q^{a_1}(R_qS_qR_q)^{a_2} R_q^{a_3}(R_qS_qR_q)^{a_4}\cdots R_q^{a_{2m-1}}(R_qS_qR_q)^{a_{2m}}
		\label{eq:pia}.
	\end{align}
	
	Here, the top left and bottom left entries of $M_q(c_1,c_2,\ldots,c_k)$ are given by $R(q)$ and $T(q)$ respectively; whereas top left  and bottom left entries of $M_q^+(a_1,a_2,\ldots,a_{2m})$ are given by $q R(q)$ and $q T(q)$ (\cite{originalconj}, \cite{leclere}). 
	
	\begin{lemma}[\cite{leclere}]Every matrix $M\in \mathrm{PSL}_q(2,\ZZ)$ can be written as $M_q(c_1,c_2,\ldots,c_k)$ for some integers $c_i\in \ZZ$. Furthermore, the trace of $M$ equals, up to a multiplicative factor of $q^{\pm N}$ for some $N\in \ZZ$ the trace of some $M_q(c_1,c_2,\ldots,c_k)$ with all $c_i\geq 2$.
	\end{lemma}
	
	This lemma implies that, especially when considering Conjecture~\ref{conj:leclere}, it makes sense to limit our attention to the case $r/t\geq 1$, which is exactly when all $c_i$ are $\geq 2$.
	
	 The traces of matrices in $\mathrm{PSL}_q(2,\ZZ)$ also have the following properties in common with circular rank polynomials (\cite{leclere}, Lemmas 3.8, 3.9, 3.10):
	\begin{itemize}[noitemsep]
		\item $\operatorname{tr}(M_q(c_1,c_2,\ldots,c_k))$ is a symmetric polynomial.
		\item It is invariant under cyclic shifts of $\llbracket c_1,c_2,\ldots,c_k\rrbracket$ so that:  $$\operatorname{tr}(M_q(c_1,c_2,\ldots,c_k))=\operatorname{tr}(M_q(c_2,c_3,\ldots,c_k,c_1)).$$
		\item When all $c_i$ are at least $2$, the polynomial $\operatorname{tr}(M_q(c_1,c_2,\ldots,c_k))$ has positive integer coefficients.
	\end{itemize}

	 \begin{example}\label{ex:32/9matrices} We have seen in Example~\ref{ex:32/9} that $32/9$ has the continued fraction expressions $[3,1,1,4]$ and $\llbracket 4,3,2,2,2\rrbracket$. The matrices corresponding to these expressions are:
	 \begin{align*}
	     	M_q(4,3,2,2,2)&
	     	=\begin{bmatrix}
	     	    1\miniplus 3q\miniplus 5q^2\miniplus 6q^3\miniplus 6q^4\miniplus 5q^5\miniplus 3q^6\miniplus 2q^7\miniplus q^8& -q(1\miniplus 3q\miniplus 5q^2\miniplus 5q^3\miniplus 5q^4\miniplus 3q^5\miniplus q^6)\\
1\miniplus 2q\miniplus 2q^2\miniplus 2q^3\miniplus q^4\miniplus q^5& -q(1\miniplus 2q\miniplus 2q^2\miniplus q^3\miniplus q^4)
	     	\end{bmatrix} \\ &=\begin{bmatrix}
\rank((2,1,1,3);q) & -q \rank((2,1,1,2);q) \\ \rank((1,3);q) & - q\rank((1,2);q).\\
	     	\end{bmatrix}, 
	     		\\M^+_q(3,1,1,4)
	     		&=\begin{bmatrix}
	     		q(1\miniplus 3q\miniplus 5q^2\miniplus 6q^3\miniplus 6q^4\miniplus 5q^5\miniplus 3q^6\miniplus 2q^7\miniplus q^8)& 1\miniplus 2q\miniplus 2q^2\miniplus q^3\miniplus q^4\\q(1\miniplus 2q\miniplus 2q^2\miniplus 2q^3\miniplus q^4\miniplus q^5)&1\miniplus q
	     		\end{bmatrix}
	     		\\ &=\begin{bmatrix}
q \rank((2,1,1,3);q) & \rank((2,1);q) \\ q \rank((1,3);q) & \rank((1);q)\\
	     	\end{bmatrix}.  
	 \end{align*}
The traces of the two matrices are given by circular rank polynomials:
\begin{align*}
    \operatorname{tr}(M_q(4,3,2,2,2))&=\rank((2,1,1,3);q)-q\rank((1,2);q)=\overline{\rank}((2,1,1,4);q)),\\
        \operatorname{tr}(M^+_q(3,1,1,4))&=q\rank((2,1,1,3);q)-1-q=\overline{\rank}((3,1,1,4);q)).\\
\end{align*}.
	 \end{example}

In Section~\ref{sec:fence}, we will show that the traces of the matrices $M_q$ and $M^+_q$ are always given by circular rank polynomials.
	
	\section{Oriented Posets} \label{sec:oriented posets}
	
	We define an \definition{oriented poset} to be a triple $\overrightarrow{P}:=(P,x_L,x_R)$ where $P$ is a poset with two specialized nodes $x_L$, $x_R$. One can think of them as the left end node and the right end node respectively. We will connect oriented posets via connecting the left end of one to the right end of the other by adding a relation. More precisely speaking, we add oriented posets via the operation:
\begin{align*}
    (P,x_L,x_R)\diplus(Q,y_L,y_R):=(S,x_L,y_R)
\end{align*} where $S$ is the unique poset given by connecting the posets $P$ and $Q$ by setting $x_R \succeq y_L$. We also use the notation $\close(\overrightarrow{P})$ to denote the structure obtained by adding the relation $x_R \succeq x_L$. This is generally a poset, though in the case  $x_R \preceq x_L$ in the original poset we get some degeneracy.

	There are two natural ways to see the length $n$ chain poset as an oriented poset: The increasing chain $\overrightarrow{U_n}$ with $x_L$ taken to be the minimal element and $x_R$ the maximum, and the decreasing chain $\overrightarrow{D_n}$ with the opposite orientation. See Figure \ref{fig:UD} below for an example.

	\begin{figure}[ht]
		\centering
		\begin{tikzpicture}[scale=.25]
			
			\fill(0,0) circle(.2);
			\fill(2,2) circle(.2);
			\fill(4,4) circle(.2);
			\fill(6,6) circle(.2);
			\fill(11,6) circle(.2);
			\fill(13,4) circle(.2);
			\fill(15,2) circle(.2);
			\fill(17,0) circle(.2);
			\fill(25,0) circle(.2);
			\fill(27,2) circle(.2);
			\fill(29,4) circle(.2);
			\fill(31,6) circle(.2);
			\fill(33,4.5) circle(.2);
			\fill(35,3) circle(.2);
			\fill(37,1.5) circle(.2);
			\fill(39,0) circle(.2);
			\node at (-0.3, 0.8) {$x_L$};
			\node at (6, 6.7) {$x_R$};
			\node at (25-0.3, 0.8) {$x_L$};
			\node at (39.3, .8) {$y_R$};
			\node  at (11, 6.7) {$y_L$};
			\node  at (17.3, 0.8) {$y_R$};
			\node  at (3, -1) {$\overrightarrow{U_3}$};
			\node  at (14, -1) {$\overrightarrow{D_3}$};
			\draw (0,0)--(6,6) (11,6)--(17,0) (39,0)--(33,4.5) (31,6)--(25,0);
			\draw[ultra thick,red,dotted](33,4.5)--(31,6);
			\node  at (32, -1) {$\overrightarrow{U_3}\diplus\overrightarrow{D_3}$};
		\end{tikzpicture}
		\caption{ $\protect \overrightarrow{U_3}\diplus\protect\overrightarrow{D_3}$ gives us the fence poset for $(3,4)$.}\label{fig:UD}		
	\end{figure}

	The motivation behind oriented posets is to build up larger families of posets modularly from smaller pieces, using the addition operation connecting the right end of one poset to the left end of another. In particular, we will use them to build regular and circular fence posets and calculate their rank polynomials in Section \ref{sec:fence}. Though this building method is not sufficient to capture the complicated nature of a general poset, it works well when the underlying graph naturally comes with edge-cuts that makes the structure disconnected, as in the cases of the Dynkin diagrams for Lie groups. With this paradigm, we can use $2\times2$ matrices to calculate the rank polynomials of the larger posets we obtain, where addition and connecting the ends operations correspond to matrix multiplication and taking the trace respectively.
	
	Consider the subsets of the distributive lattice $J(P)$  determined by the inclusion or exclusion of vertices $x_L$ and $x_R$. Our notation is that the subscript on the left describes the exclusion ($0$) or inclusion ($1$) of $x_L$ and the subscript on the right describes the inclusion or exclusion of $x_R$: 
	\begin{align*}
		\rank(\overrightarrow{P};q):=& \sum_{I\in J(P)}  q^{|I|}  \qquad \qquad  &_0\rank(\overrightarrow{P};q):= \sum_{\substack{I\in J(P)\\ x_L\notin I} } q^{|I|} \\
		\rank_0(\overrightarrow{P};q):=& \sum_{\substack{I\in J(P)\\ x_R\notin I} } q^{|I|}  \qquad &
		_0\rank_0(\overrightarrow{P};q):= \sum_{\substack{I\in J(P)\\ x_L, x_R\notin I}}  q^{|I|} 
	\end{align*}
	The superscripts can also be taken to be $1$, to specify $x_L$ or $x_R$ is going to be included.

	\begin{defn} We define the \definition{rank matrix} of an oriented poset $(P,x_L,x_R)$ as follows:
		$$\displaystyle \rmm_q(\overrightarrow{P}):=\begin{bmatrix} \rank(\overrightarrow{P};q) & -\rank_1(\overrightarrow{P};q)\\ _0\rank(\overrightarrow{P};q) & -\,_0\rank_1(\overrightarrow{P};q)
		\end{bmatrix}
		$$
	\end{defn}
	
	The increasing and decreasing length $n$ chains have the following rank matrices:
	\begin{eqnarray*}
		&\rmm_q(\overrightarrow{U_n})&=\begin{bmatrix} [n+2]_q &-q^{n+1}\\1& 0
		\end{bmatrix}  =R_q^{n+2}S_q.\\
		&\rmm_q(\overrightarrow{D_n})&=
		\begin{bmatrix} [n+2]_q & -q[n+1]_q \\ [n+1]_q & -q[n]_q
		\end{bmatrix}=(R_q^2S_q)^{n+1}.
	\end{eqnarray*}
	Note that in both cases, the resulting matrix lies in $\mathrm{PSL_q}(2,\ZZ)$.
	For simplicity, we will use the notation $D_0$ for the  matrix corresponding to the case $\overrightarrow{\bullet}=\overrightarrow{U_0}=\overrightarrow{U_0}$ of having only one vertex. This can be seen as the generator of a down-step.
	\begin{align*}
		D_0:=&\rmm_q(\overrightarrow{\bullet})=\begin{bmatrix} [2]_q & -q\\1 & 0
		\end{bmatrix}= R_q^{2}S_q.
	\end{align*}

	As mentioned above, the motivation behind rank matrices is to be able to calculate the rank polynomials corresponding to $\close(\overrightarrow{P})$ and $\overrightarrow{P}\diplus\overrightarrow{Q}$ using simple matrix operations.
	
	\begin{prop} Let $\overrightarrow{P}=(P,x_L,x_R)$ and  $\overrightarrow{Q}=(Q,y_L,y_R)$ be two oriented posets. Then we have:
		\begin{eqnarray}
			&\rmm_q(\overrightarrow{P}\diplus\overrightarrow{Q})= \rmm_q(\overrightarrow{P})\cdot \rmm_q(\overrightarrow{!})\\
			&\rank(\close(\overrightarrow{P}))=\mathrm{tr}(\rmm_q(\overrightarrow{P}))
		\end{eqnarray}
	\end{prop}
	\begin{proof} The way we form $\overrightarrow{P}=(P,x_L,x_R)\diplus\overrightarrow{Q}=(Q,y_L,y_R)$  is to add a new relation $x_R \succeq y_L$. So any ideal of $\overrightarrow{P}\diplus\overrightarrow{Q}$ can be seen as a pair $(I_1,I_2)$ where $I_1$ is an ideal of $P$, $I_2$ is an ideal of $Q$ and we have $x_R\in I_1\Rightarrow y_L\in I_2$.
		On the rank polynomial side, this means:
		\begin{eqnarray*}
			&\rank(\overrightarrow{P}\diplus\overrightarrow{Q};q)= \rank(\overrightarrow{P};q)\, \rank(\overrightarrow{Q};q)-\rank_1(\overrightarrow{P};q)\,_0\rank(\overrightarrow{Q};q).
		\end{eqnarray*}
		Note that this is exactly the upper left entry of $\rmm_q(\overrightarrow{P})\cdot \rmm_q(\overrightarrow{q})$. Other entries are calculated by adding the extra conditions $x_L$ is in and $y_R$ is out to the above equation, and similarly match the matrix multiplication.
		
		For the second equation, note that the trace of $\rmm_q(\overrightarrow{P})$ is given by $\rank(\overrightarrow{P};q)-\,_0\rank_1(\overrightarrow{P};q)$, which counts the ideals with $x_L\in I \Rightarrow x_R\in I$. This is the same thing as adding the relation $x_L \succeq x_R$.
	\end{proof}
	
	\begin{rmk} The down chain $\overrightarrow{D_n}$ can be built by adding up $n+1$ copies of $\overrightarrow{\bullet}$. Accordingly we have:
		\begin{align}
			\rmm_q(\overrightarrow{D_n})=&(D_0)^{n+1} \label{eq:DtoT0}
		\end{align}
	\end{rmk}
	\begin{example} For the poset $(3,4)$ given in Figure \ref{fig:UD}, the rank polynomial is given by the top left entry of
		\begin{align*}
			\displaystyle \rmm_q(\overrightarrow{U_3})\,D_0^4=&
			\begin{bmatrix} ([5]_q)^2-q^4[4]_q & q^5 [3]_q-q [5]_q [4]_q \\ [5]_q & -q[4]_q\,
			\end{bmatrix}
		\end{align*}
		\begin{align*}
			\rank((3,4);q)=&[5]_q^2-q^4[4]_q=1+2q+3q^2+4q^3+4q^4+3q^5+2q^6+q^7+q^8.  
		\end{align*}
	\end{example}
	
	\section{Dual Addition}\label{sec:dual}
	
	In this section, we look at a dual definition of addition where we add the relation $x_R \preceq y_L$ instead of $x_R \succeq y_L$ to connect two posets $(P,x_L,x_R)$ and $(Q,y_L,y_R)$. We denote the directed poset $(S,x_L,y_R)$ we obtain this way by $(P,x_L,x_R)\revdiplus (Q,y_L,y_R)$.We also use the notation $\revclose(\overrightarrow{P})$ to denote the structure obtained by adding the relation $x_R \preceq x_L$ to $(P,x_L,x_R)$. 

	\begin{figure}[ht]
		\centering
		\begin{tikzpicture}[scale=.25]
			\draw[ultra thick,red,dotted](31,4.5)--(33,6);
			\fill(25,0) circle(.2);
			\fill(27,1.5) circle(.2);
			\fill(29,3) circle(.2);
			\fill(31,4.5) circle(.2);
			\fill(33,6) circle(.2);
			\fill(35,4) circle(.2);
			\fill(37,2) circle(.2);
			\fill(39,0) circle(.2);
			\node at (25-0.3, 0.8) {$x_L$};
			\node at (39.3, .8) {$x_R$};
			\draw (25,0)--(31,4.5) (33,6)--(39,0);
		\end{tikzpicture}$\qquad \qquad$\begin{tikzpicture}[scale=.25]
			\draw[ultra thick,red,dotted,->](23,1)--(25,2);
			\draw[ultra thick,red,dotted,->](39,0)--(41,1);
			\fill(25,2) circle(.2);
			\fill(27,3) circle(.2);
			\fill(29,4) circle(.2);
			\fill(31,5) circle(.2);
			\fill(33,6) circle(.2);
			\fill(35,4) circle(.2);
			\fill(37,2) circle(.2);
			\fill(39,0) circle(.2);
			\draw (25,2)--(33,6)--(39,0);
		\end{tikzpicture}
		\caption{The posets $\protect\overrightarrow{U_3} \revdiplus \protect\overrightarrow{D_3}$ (left) and $\revclose (\protect\overrightarrow{U_3} \revdiplus \protect\overrightarrow{D_3} )$  (right).}
		\label{fig:dualadding}
	\end{figure}

	\begin{defn} The \definition{dual rank matrix} of a oriented poset $(P,x_L,x_R)$ is defined as follows:
		\begin{align}
			\displaystyle \drm_q(\overrightarrow{P}):=\begin{bmatrix} \rank_1(\overrightarrow{P};q) & \rank_0(\overrightarrow{P};q)\\ _0\rank_1(\overrightarrow{P};q) & _0\rank_0(\overrightarrow{P};q)
			\end{bmatrix}
			\label{eq:Dualdef}   
		\end{align}
	\end{defn}

	\begin{prop} \label{prop:dualadding}Let $\overrightarrow{P}=(P,x_L,x_R)$ and  $\overrightarrow{Q}=(Q,y_L,y_R)$ be two oriented posets. Then we have:
		\begin{eqnarray}
			&\drm_q(\overrightarrow{P}\revdiplus\overrightarrow{Q})= \drm_q(\overrightarrow{P})\cdot \drm_q(\overrightarrow{Q}),\\
			&R(\revclose(\overrightarrow{P}))=\mathrm{trace}(\drm_q(\overrightarrow{P})).
		\end{eqnarray}
	\end{prop}
	
The proof is similar to the regular case and will be omitted.	
	
	For the chains $\overrightarrow{U_n}$ and $\overrightarrow{D_n}$, we get the following expressions for dual rank matrices:
	\begin{align}
		\drm_q(\overrightarrow{U_n})&=\begin{bmatrix}
			q^{n+1}&[n+1]_q\\
			0&1
		\end{bmatrix}=R_q^{n+1}\in \mathrm{PSL_q}(2,\ZZ),\\
		\drm_q(\overrightarrow{D_n})&=\begin{bmatrix}
			q[n+1]_q&1\\
			q[n]_q&1
		\end{bmatrix}=R_q(R_q S_q R_q)^n\in \mathrm{PSL_q}(2,\ZZ).
	\end{align}
	We use the notation $U_0$ for the dual rank matrix  of $\overrightarrow{\bullet}=\overrightarrow{U_0}=\overrightarrow{D_0}$:
	\begin{align*}
		U_0&=\begin{bmatrix}
			q&1\\
			0&1
		\end{bmatrix}=R_q.
	\end{align*}
	\begin{rmk} The increasing chain $\overrightarrow{U_n}$ can be built by adding up $n+1$ copies of $\overrightarrow{\bullet}$ in this dual way. Accordingly we have:
		\begin{align}
			\drm_q(\overrightarrow{U_n})=&(U_0)^{n+1}. \label{eq:UtoT0}
		\end{align}
	\end{rmk}
	
	\begin{example} As seen in Figure \ref{fig:dualadding} above,  $\overrightarrow{U_3}\revdiplus\overrightarrow{D_3}$ gives us the fence poset $F(4,3)$ and  $\revclose(\overrightarrow{U_3}\revdiplus\overrightarrow{D_3})$ gives us the circular fence poset $\overline{F}(5,3)$. The corresponding rank polynomials can be calculated as follows:
		\begin{align*}
			\drm_q(\overrightarrow{U_3}\revdiplus\overrightarrow{D_3})&=\begin{bmatrix}
				q^{4}&[4]_q\\
				0&1
			\end{bmatrix}\cdot\begin{bmatrix}
				q[4]_q&1\\
				q[3]_q&1
			\end{bmatrix}= \begin{bmatrix}
				(q+q^2+q^3+q^5)[4]_q&[5]_q\\
				q[3]_q&1
			\end{bmatrix}\\
			\rank((4,3);q)&=\drm_q(\overrightarrow{U_3}\revdiplus\overrightarrow{D_3})[1,1]+\drm_q(\overrightarrow{U_3}\revdiplus\overrightarrow{D_3})[1,2]
			\\&=(q+q^2+q^3+q^5)[4]_q+[5]_q=[5]_q[4]_q+q^8.\\
			\overline{\rank}((5,3);q)&=\mathrm{tr}(\overrightarrow{U_3}\revdiplus\overrightarrow{D_3}))=1+q+2q^2+3q^3+3q^4+3q^5+2q^6+q^7+q^8.
		\end{align*}
		
	\end{example}
	We will use the expression \emph{rank matrices} of $\overrightarrow{P}$ when we need to refer to the rank matrix $\rmm_q(\overrightarrow{P})$ and the dual rank matrix $\drm_q(\overrightarrow{P})$ together.	
	
	\begin{lemma} The rank matrices of an oriented poset $\overrightarrow{P}$ are related by:\label{lem:dualregularconnection}
		
		\begin{align*}
			\drm_q(\overrightarrow{P})&=\rmm_q(\overrightarrow{P})\begin{bmatrix}
				0&1\\-1&1
			\end{bmatrix}\qquad \qquad   \rmm_q(\overrightarrow{P})=\drm_q(\overrightarrow{P})\begin{bmatrix}
				1&-1\\ 1&0
			\end{bmatrix}.
		\end{align*}
		As a consequence, for any two oriented posets $\overrightarrow{P}=(P,x_L,x_R)$ and  $\overrightarrow{Q}=(Q,y_L,y_R)$ we have:
		\begin{align}
			\drm_q(\overrightarrow{P})\cdot \rmm_q(\overrightarrow{Q})&=\rmm_q(\overrightarrow{P}\revdiplus\overrightarrow{Q}), \\
			\rmm_q(\overrightarrow{P})\cdot \drm_q(\overrightarrow{Q})&=\drm_q(\overrightarrow{P}\diplus\overrightarrow{Q}).
		\end{align}
	\end{lemma}
	
	\begin{proof}
		The first part follows from the fact that $\rank(q)=\rank_0(q)+\rank_1(q)$ and $_0\rank(q)=\,  _0\rank_0(q)+ \,_0\rank_1(q)$ and the second part follows from the associativity of matrix multiplication.
	\end{proof}

	\section{The Case of Fence Posets}\label{sec:fence}
	
	As we are able to build up any fence poset via adding up up and down steps, we can use the corresponding matrices to calculate rank polynomials. For a composition $\alpha$ of $n$, we can see the fence poset as an oriented poset by taking $x_0$ to be the left end, and $x_n$ to be the right end. One can build the fence poset of $\alpha$ from up and down steps. Looking at the corresponding multiplication of the matrices $U_0$ and $D_0$ gives us the rank matrices for $\overrightarrow{\alpha}$. Which matrix we obtain is determined by whether we pick $U_0$ or $D_0$ for the rightmost node.
	
	\begin{prop} Consider the oriented poset $\overrightarrow{\alpha}$ corresponding to $\alpha=(u_1, d_1, u_2,d_2,\ldots,u_s,d_s)$ where we allow $d_s=0$ for full generality. 
	Then $\overrightarrow{\alpha}$ has rank matrices:
	\begin{align}
	    \rmm(\overrightarrow{\alpha})&:=(U_0)^{u_1}(D_0)^{d_1}(U_0)^{u_2}(D_0)^{d_2}\cdots(U_0)^{u_{s-1}}(D_0)^{d_{s-1}}(U_0)^{u_s}(D_0)^{d_s+1},\label{eq:fen2}\\
	    \drm(\overrightarrow{\alpha})&:=(U_0)^{u_1}(D_0)^{d_1}(U_0)^{u_2}(D_0)^{d_2}\cdots(U_0)^{u_{s-1}}(D_0)^{d_{s-1}}(U_0)^{u_s}(D_0)^{d_s}(U_0),\label{eq:fen22}.
	\end{align}
	The rank polynomial $\rank(\alpha,q)$ is given by the top left entry of $\rmm(\overrightarrow{\alpha})$. Alternatively, one can take the sum of the entries on the top row of $\drm(\overrightarrow{\alpha})$.\label{prop:fence}
	\end{prop}

	\begin{example}We can construct the fence poset for $\alpha=(2,1,1,3)$ from Figure \ref{fig:2113} by connecting $\overrightarrow{U_{2}}\diplus\overrightarrow{U_{1}}\diplus\overrightarrow{D_{2}}$ or equivalently $ \overrightarrow{\bullet}\revdiplus\overrightarrow{\bullet}\revdiplus\overrightarrow{\bullet}\diplus\overrightarrow{\bullet}\revdiplus\overrightarrow{\bullet}\diplus\overrightarrow{\bullet}\diplus\overrightarrow{\bullet}\diplus\overrightarrow{\bullet}$.
		\begin{figure}[ht]
			\centering
			\begin{tikzpicture}[scale=.6]
				\draw[ultra thick,red,dotted] (2,3)--(3,2) (4,3)--(5,7/3);
				\fill(0,1) circle(.1);
				\fill(1,2) circle(.1);
				\fill(2,3) circle(.1);
				\fill(3,2) circle(.1);
				\fill(4,3) circle(.1);
				\fill(5,7/3) circle(.1);
				\fill(6,5/3) circle(.1);
				\fill(7,1) circle(.1);
				\draw (0,1)--(2,3) (3,2)--(4,3) (5,7/3)--(7,1);
				\node at (-0.2,1.5) {$x_L$};
				\node at (7.2, 1.5) {$x_R$};
				\node at (3.5,.2) {$\overrightarrow{U_{2}}\diplus\overrightarrow{U_{1}}\diplus\overrightarrow{D_{2}}$};
			\end{tikzpicture}\raisebox{.7cm}{$\quad =\quad$}
			\begin{tikzpicture}[scale=.7]
				\draw[ultra thick,red,dotted] (2,3)--(3,2) (4,3)--(7,1);
				\draw[ultra thick,blue,dotted] (0,1)--(2,3) (3,2)--(4,3);
				\fill(0,1) circle(.1);
				\fill(1,2) circle(.1);
				\fill(2,3) circle(.1);
				\fill(3,2) circle(.1);
				\fill(4,3) circle(.1);
				\fill(5,7/3) circle(.1);
				\fill(6,5/3) circle(.1);
				\fill(7,1) circle(.1);
				\node at (-0.2,1.5) {$x_L$};
				\node at (7.2, 1.5) {$x_R$};
				\node at (3.5,.2) { $\overrightarrow{\bullet}\revdiplus\overrightarrow{\bullet}\revdiplus\overrightarrow{\bullet}\diplus\overrightarrow{\bullet}\revdiplus\overrightarrow{\bullet}\diplus\overrightarrow{\bullet}\diplus\overrightarrow{\bullet}\diplus\overrightarrow{\bullet}$};
			\end{tikzpicture}
			\label{fig:2113connect}
		\end{figure}
		The corresponding rank matrix is given by:
		\begin{align*}&\rmm(\overrightarrow{(2,1,1,3)})=
			\begin{bmatrix}
				[4]_q&-q^3\\1&0 
			\end{bmatrix}
			\begin{bmatrix}
				[3]_q&-q^2\\1&0 
			\end{bmatrix}
			\begin{bmatrix}
				[4]_q&-q [3]_q\\ [3]_q&-q[2]_q 
			\end{bmatrix}=U_0^2 D_0 U_0 D_0^4\\
			&=    \begin{bmatrix}
1\miniplus 3q\miniplus 5q^2\miniplus 6q^3\miniplus 6q^4\miniplus 5q^5\miniplus 3q^6\miniplus 2q^7\miniplus q^8& -q(1\miniplus 3q\miniplus 5q^2\miniplus 5q^3\miniplus 5q^4\miniplus 3q^5\miniplus q^6)\\
1\miniplus 2q\miniplus 2q^2\miniplus 2q^3\miniplus q^4\miniplus q^5& -q(1\miniplus 2q\miniplus 2q^2\miniplus q^3\miniplus q^4)
			\end{bmatrix}
		\end{align*}
		Note that the top left hand side evaluates to $\rank((2,1,1,3);q)$ that we calculated in Example \ref{ex:rankpolyfor2113}. Furthermore, the whole matrix matches $M_q(4,3,2,2,2)$ from Example~\ref{ex:32/9matrices}.
	\end{example}
	
	After building up the fences this way, we can use the $\close\,$ operation to make our fences circular. One should note that the operation $\close$ adds a new relation $x_R\succeq x_L$ to our poset. If $\alpha$ has an even number of parts, this corresponds to increasing the last part by $1$, whereas if it has an odd number of parts (larger than $1$) an extra part of size $1$ is added. In any case we end up with an even number of parts in our circular fence. That is why the matrix for the circular fence that we define in Equation \ref{eq:circfen2} below differs from the one in  Equation \ref{eq:fen2} in the last entry.  
	
	\begin{prop} Let $P$ be the circular fence poset corresponding $\alpha =(u_1, d_1, u_2,d_2,\ldots,u_s,d_s)$. The circular rank polynomial $\overline{\rank}(\alpha,q)$ is given by the trace of the following matrix:\label{prop:circfence}
		
		\begin{align}
			\clmat_q(\alpha)&:=(U_0)^{u_1}(D_0)^{d_1}(U_0)^{u_2}(D_0)^{d_2}\cdots(U_0)^{u_s}(D_0)^{d_s}\label{eq:circfen2}\\&=\rmm_q(\overrightarrow{\alpha})\cdot D_0^{-1}=\rmm_q(\overrightarrow{\alpha})\cdot U_0^{-1}.\tag*{}
		\end{align} 
	\end{prop}

	\begin{example}For the circular fence poset corresponding to $\alpha=(2,1,1,3)$ from Figure \ref{fig:2113} we connect the two ends of $\overrightarrow{U_{2}}\diplus\overrightarrow{U_{1}}\diplus\overrightarrow{D_{1}}$. We can equivalently look at the trace of the matrix $\clmat_q((2,1,1,3))=U_0^2 \cdot D_0 \cdot U_0 \cdot D_0^2$. Note that taking the trace gives us the final down-step connecting the two ends.
		
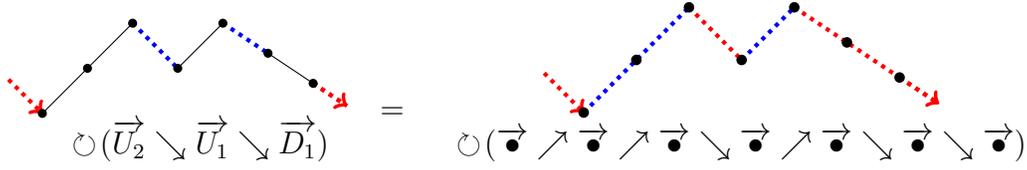
\begin{figure}[ht]
			\centering
			\begin{tikzpicture}[scale=.6]
				\draw[ultra thick,blue,dotted] (2,3)--(3,2) (4,3)--(5,7/3);
				\draw[ultra thick,red,dotted,->] (6,5/3)--(6.75,7/6) (-.75,1.75)--(0,1);
				\draw[ultra thick,red,dotted,->](6,5/3)--(6.75,7/6);
				\fill(0,1) circle(.1);
				\fill(1,2) circle(.1);
				\fill(2,3) circle(.1);
				\fill(3,2) circle(.1);
				\fill(4,3) circle(.1);
				\fill(5,7/3) circle(.1);
				\fill(6,5/3) circle(.1);
				\draw (0,1)--(2,3) (3,2)--(4,3) (5,7/3)--(6,5/3) ;
				\node at (3.5,.4) {$\close(\overrightarrow{U_{2}}\diplus\overrightarrow{U_{1}}\diplus\overrightarrow{D_{1}})$};
			\end{tikzpicture}\raisebox{.7cm}{$\quad =\quad$}
			\begin{tikzpicture}[scale=.7]
				\draw[ultra thick,red,dotted] (2,3)--(3,2) (4,3)--(6,5/3);
				\draw[ultra thick,red,dotted,->] (6,5/3)--(6.75,7/6) (-.75,1.75)--(0,1);
				\draw[ultra thick,red,dotted,->](6,5/3)--(6.75,7/6);
					\draw[ultra thick,blue,dotted] (0,1)--(2,3) (3,2)--(4,3);
				\fill(0,1) circle(.1);
				\fill(1,2) circle(.1);
				\fill(2,3) circle(.1);
				\fill(3,2) circle(.1);
				\fill(4,3) circle(.1);
				\fill(5,7/3) circle(.1);
				\fill(6,5/3) circle(.1);
				\node at (3,.4) {$\close(\overrightarrow{\bullet}\revdiplus\overrightarrow{\bullet}\revdiplus\overrightarrow{\bullet}\diplus\overrightarrow{\bullet}\revdiplus\overrightarrow{\bullet}\diplus\overrightarrow{\bullet}\diplus\overrightarrow{\bullet})$};
			\end{tikzpicture}
			\label{fig:2113connect2}
		\end{figure}
		\begin{align*}&\clmat_q((2,1,1,3))=
			\begin{bmatrix}
				[4]_q&-q^3\\1&0 
			\end{bmatrix}
			\begin{bmatrix}
				[3]_q&-q^2\\1&0 
			\end{bmatrix}
			\begin{bmatrix}
				[3]_q&-q [2]_q\\ [2]_q&-q[1]_q 
			\end{bmatrix}\\
			&=    \begin{bmatrix}
				1+3q+5q^2+6q^3+6q^4+5q^5+3q^6+2q^7+q^8&-q[3]_q(1+q)(1+q+q^3)\\ 1+2q+2q^2+q^3+q^4&-q(1+2q+q^2+q^3).
			\end{bmatrix}
		\end{align*}
		
		The trace of the matrix is equal to $[5]_q[4]_q=\overline{\rank}((2,1,1,3);q)$ from Example \ref{ex:rankpolyfor2113}.
	\end{example}
	\begin{prop}\label{T:1} Let $\llbracket c_1,c_2,\ldots,c_k\rrbracket$ and $[a_1,a_2,\cdots,a_{2m}]$ be the continued fraction expressions for $r/t\geq 1$. Then we have:
	\begin{align}
	    M_q(c_1,c_2,\ldots,c_k)&=\clmat_q(\alpha_{r/q}) \cdot D_0=\rmm_q(\overrightarrow{\alpha_{r/q}}), \label{eq:Mq}\\
	    M^+_q(a_1,a_2,\cdots,a_{2m})&
	=\clmat_q(\alpha_{r/q}) \cdot D_0 \cdot U_0=\drm_q(\overrightarrow{(a_1-1,a_2,\ldots,a_{2m},1)}))\label{eq:M+q}.
	\end{align}
	Furthermore,
		\begin{align}
			\mathrm{tr}(M_q(c_1,c_2,\ldots,c_k))&=\overline{\rank}((a_1-1,a_2,\ldots,a_{2m};q),\label{eq:traceM}\\
	\mathrm{tr}(M_q^+(a_1,a_2,\ldots,a_{2m}))&=\overline{\rank}((a_1,a_2,\ldots,a_{2m});q)\label{eq:traceM+}.
		\end{align}
	\end{prop}

	\begin{proof} Remember that by Equation~\ref{eq:atocconversion}, we have $\llbracket c_1,c_2,\ldots,c_k\rrbracket=\llbracket a_1+1,2^{a_2-1},a_3+2,2^{a_4-1},a_5+2,\ldots,a_{2m-1}+2,2^{a_{2m}-1}\rrbracket$. We plug in $D_0=R_q^2S_q$ and $U_0=R_q$ in Equations \ref{eq:pic} \ref{eq:pia} to get a rank matrix of a fence poset.
		\begin{align*}
				M_q(c_1,c_2,\ldots,c_k)&=
R_q^{a_1+1}S_q (R_q^2 S_q)^{a_2-1} R_q^{a_3+2}S_q (R_q^2 S_q)^{a_4-1}\cdots R_q^{a_{2m-1}+2} S_q(R_q^2 S_q)^{a_{2m}-1}\\
				&=(U_0)^{a_1-1}(D_0)^{a_2} (U_0)^{a_3}(D_0)^{a_4}\cdots (U_0)^{a_{2m-1}}(D_0)^{a_{2m}}\\
		&=\rmm_q(\overrightarrow{\alpha_{r/t}})),\\
			M_q^+(a_1,a_2,\ldots,a_{2m})&=R_q^{a_1}(R_qS_qR_q)^{a_2} R_q^{a_3}(R_qS_qR_q)^{a_4}\cdots R_q^{a_{2m-1}}(R_qS_qR_q)^{a_{2m}}\\
			&=R_q^{a_1-1}(R_q^2S_q)^{a_2}R_q^{a_3}(R_q^2S_q)^{a_4}\cdots (R_q)^{a_{2m-1}}(R_q^2S_q)^{a_{2m}}R_q\\
			&=(U_0)^{a_1-1}(D_0)^{a_2}(U_0)^{a_3}(D_0)^{a_4}\cdots(U_0)^{a_{2m-1}}(D_0)^{a_{2m}}(U_0)\\ &=\rmm_q(\overrightarrow{\alpha_{r/t}}))\cdot (U_0). \end{align*} The trace formula for $M_q(c_1,c_2,\ldots,c_k)$ follows directly. For the case of $M_q^+(a_1,a_2,\ldots,a_{2m})$, we take advantage of the fact that trace remains unchanged if we move $U_0$ from one side to the other.
			\begin{align*}
			\mathrm{tr}(M_q^+(a_1,a_2,\ldots,a_{2m}))&=\mathrm{tr}(\clmat_q(\alpha_{r/q}) \cdot D_0 \cdot U_0)=\mathrm{tr}(U_0 \cdot \clmat_q(\alpha_{r/q}) \cdot D_0 )\\&
			=\mathrm{tr}(\clmat_q((a_1,a_2,\ldots,a_{2m})))=\overline{\rank}((a_1,a_2,\ldots,a_{2m});q).
	\end{align*}	
 Alternatively, one can consider the continuant expressions for the matrices $M_q$ and $M_q^+$ (see Propositions 5.2 and 5.3 from \cite{originalconj}) to verify directly that the entries match those of the rank matrices.  \end{proof}
	
	The above proposition allows us to apply the discussion of unimodality of rank polynomials of circular fences to the traces of matrices in $\mathrm{PSL}_q(2,\ZZ)$. We have one degenerate case of $\llbracket c_1,c_2,\ldots,c_k\rrbracket$ with $c_i=2$ for all $i$. In that case the corresponding fence is a chain of length $k$, and taking the trace adds the relation $x_1\succeq x_k$ where   $x_2$ is the minimal node and $x_k$ is the maximal. The resulting structure has all nodes equal to each other, so the only two ideals consist of the empty set and the whole chain. Another violation of unimodality comes from the cases $(1,k,1,k)$ and $(k,1,k,1)$ whose circular rank polynomials are not unimodal.
	\begin{corollary}\label{cor:counters} The polynomials $\mathrm{tr}(M_q(c_1,c_2,\ldots,c_k))$ are not always unimodal. For $c_i\geq 2$ we get the following counterexamples:
\begin{align}
    \tag*{(CE 1)}\operatorname{tr}(M_q(2^k))=&1+q^k &\text{ where } k\geq 2,\\
    \tag*{(CE 2)}\operatorname{tr}(M_q(n+2,n+2)=&1+2q+3q^2+\cdots +(n+1)q^{n}&\\\tag*{} &+ n\,q^{n+1}+(n+1)q^{n+2}+\ldots &
     \\ \tag*{} &\ldots+3q^{2n-2}+2q^{2n-1}+ q^{2n} &\text{ where } n\geq 1,\\
    \tag*{(CE 2')} \operatorname{tr}(M_q(2^{n-1},3,2^{n-1},3))=&\operatorname{tr}(M_q(n+2,n+2)) &\text{ where }n\geq 1.
\end{align}
	\end{corollary} 
	\begin{conj} \label{con:lec+} Up to cyclic shifts of $\llbracket c_1,c_2,\ldots,c_k\rrbracket$ , the examples listed in Corollary~\ref{cor:counters} are the only counterexamples to unimodality for $\mathrm{tr}(M_q(c_1,c_2,\ldots,c_k))$, $c_i\geq 2$.
	\end{conj}
	As mentioned before in Section \ref{sec:prelim}, this actually is a statement concerning the traces of all matrices in $\mathrm{PSL}_q(2,\ZZ)$, as all such matrices are given by some $M_q(c_1,c_2,\ldots,c_k)$ with $c_i\geq 2$ up to a factor of $q^{\pm N}$.

We can also apply the partial unimodality results from Theorem~$\ref{thm:ourtheorem}$ using the rank polynomial descriptions of traces of $M_q$ and $M_q^+$. For the continued fraction expression $[a_1,a_2,\ldots,a_s]$ one only needs to check the conditions for $\alpha=(a_1,a_2,\ldots,a_s)$. For the corresponding expression $\llbracket c_1,c_2,\ldots,c_k\rrbracket$, one checks $\alpha -1$ formed by subtracting $1$ from the leftmost part of $\alpha$ instead.

	\section{Rank Polynomial Identities}\label{sec:identities}
	
In this section we will be concerned with some multiplicative identities concerning rank polynomials.

	\begin{example} Looking at the rank polynomials for $\overline{\rank}((k,2,1,1,1,1);q)$ one can spot a common factor. 
		\begin{eqnarray*}
			&\overline{\rank}((1,2,1,1,1,1;q)&=(q^6 + 2q^5 + 2q^4 + 3q^3 + 2q^2 + 2q + 1)(q + 1)\\
			&\overline{\rank}((2,2,1,1,1,1;q)&=(q^6 + 2q^5 + 2q^4 + 3q^3 + 2q^2 + 2q + 1)(q^2 + q + 1)\\
			&\overline{\rank}((3,2,1,1,1;q)&=(q^6 + 2q^5 + 2q^4 + 3q^3 + 2q^2 + 2q + 1)(q^3+q^2 + q + 1)
		\end{eqnarray*}
		That polynomial matches $\overline{\rank}((3,1,1,1;q)=(q^6 + 2q^5 + 2q^4 + 3q^3 + 2q^2 + 2q + 1)$.
		The pattern persists when we have a larger number of pieces of size $1$. 
		\begin{align*}
			\overline{\rank}((k,2,1^{2s+2});q)&=[k+1]_q\cdot \overline{\rank}((3,1^{2s+1});q).\tag{Id 0.1}\\
		\end{align*}
		We also have a companion identity when the $3$ on the right hand side varies.
		\begin{align*}
			\overline{\rank}((k+1,1,k,1^{2s+1});q)&=[k+1]_q\cdot \overline{\rank}((k+2,1^{2s+1});q).\tag{Id 0.2}
		\end{align*}
	\end{example}
	\begin{figure}[ht]
		\centering
		\begin{tikzpicture}[scale=.3]
			\fill(-6, 0) circle(.2);
			\fill (-5.5, 1)circle(.2);
			\fill(-5, 2) circle(.2);
			\fill(-4,4) circle(.2);  
			\fill(-4.5,3) circle(.2);
			\fill(-3,2) circle(.2);
			\fill(-2,0) circle(.2);
			\fill(-1,2) circle(.2); 
			\fill(0,0) circle(.2); 
			\fill(1,2) circle(.2); 
			\fill(2,0) circle(.2);
			\fill(3,2) circle(.2); 
			\fill(4,0) circle(.2); 
			\fill(5,2) circle(.2); 
			\node (0) at (-6, 0) {};
			\node (1) at (-5.5, 1) {};
			\node (2) at (-5, 2) {};
			\node (3) at (-4, 4) {};
			\node (4) at (-4.5, 3) {};
			\node (5) at (-3, 2) {};
			\node (6) at (-2, 0) {};
			\node (7) at (-1, 2) {};
			\node (8) at (0, 0) {};
			\node (9) at (1, 2) {};
			\node (10) at (2, 0) {};
			\node (11) at (3, 2) {};
			\node (12) at (4, 0) {};
			\node (13) at (5, 2) {};
			\draw (0.center) to (3.center);
			\draw (3.center) to (6.center);
			\draw (6.center) to (7.center);
			\draw (7.center) to (8.center);
			\draw (8.center) to (9.center);
			\draw (9.center) to (10.center);
			\draw (10.center) to (11.center);
			\draw (11.center) to (12.center);
			\draw (12.center) to (13.center);
			\draw[ultra thick, red,dotted,->] (5,2)--(6.5,.5);
			\draw[ultra thick, red,dotted,->] (-7.5,1.5)--(-6,0);
		\end{tikzpicture}\raisebox{.5cm}{$\qquad \longleftrightarrow$}\begin{tikzpicture}[scale=.3]
			\fill(-6, 0) circle(.2);
			\fill(-1.5,2) circle(.2);
			\fill(-2,1) circle(.2);
			\fill(-2.5,0) circle(.2);
			\fill(-1,3) circle(.2); 
			\fill(0,0) circle(.2); 
			\fill(1,2) circle(.2); 
			\fill(2,0) circle(.2);
			\fill(3,2) circle(.2); 
			\fill(4,0) circle(.2); 
			\fill(5,2) circle(.2); 
			\fill(-6, 1) circle(.2);
			\fill (-6, 2)circle(.2);
			\fill (-6,3)circle(.2);
			\node (7) at (-1, 3) {};
			\node (8) at (0, 0) {};
			\node (9) at (1, 2) {};
			\node (10) at (2, 0) {};
			\node (11) at (3, 2) {};
			\node (12) at (4, 0) {};
			\node (13) at (5, 2) {};
			\node (14) at (-1.5, 2) {};
			\node (15) at (-2, 1) {};
			\node (16) at (-2.5, 0) {};
			\node (17) at (-9.75, 7) {};
			\draw (-6,0)--(-6,3);
			\draw (7.center) to (8.center);
			\draw (8.center) to (9.center);
			\draw (9.center) to (10.center);
			\draw (10.center) to (11.center);
			\draw (11.center) to (12.center);
			\draw (12.center) to (13.center);
			\draw (7.center) to (16.center);
			\draw[ultra thick, red,dotted,->] (5,2)--(6.5,.5);
			\draw[ultra thick, red,dotted,->] (-4,1.5)--(-2.5,0);
		\end{tikzpicture}
		
		\caption{Illustration of (Id 0.1) at $s=4$, $k=3$}
		\label{fig:example1}
	\end{figure}
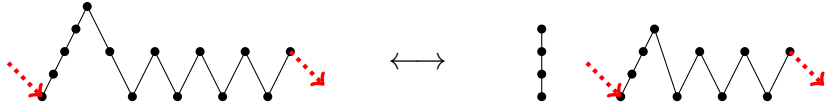

	We will see next that the identities described are actually part of a larger family of rank polynomial identities concerning palindromic compositions. For this purpose, we extend the notation $\clmat_q(\alpha):=(U_0)^{u_1}(D_0)^{d_1}(U_0)^{u_2}(D_0)^{d_2}\cdots(U_0)^{u_s}(D_0)^{d_s}$ from Equation~\ref{eq:circfen2} to include the cases where $u_1$, $d_s$ or both may be zero. In the language of fences, this means that we include the cases of starting with a down step as well. See Section~\ref{sec:prelim} for a discussion on the corresponding rank polynomials.
	
	\begin{lemma} \label{lem:palin}Let $\rho=( \rho_1,\rho_2,\ldots,\rho_s)$ be a palindromic composition (A composition satisfying $\rho_1=\rho_s$, $\rho_2=\rho_{s-1}$, $\ldots$ etc.).
	Consider the following maps on $2\times2$ matrices:
	\begin{align*}
	    \Psi^+(X)&:=\operatorname{tr}\left(\begin{bmatrix}1& -1-q\\1&-1
	    \end{bmatrix}\cdot \begin{bmatrix}x_{11}& x_{12}\\x_{21}&x_{22}
	    \end{bmatrix}\right)=x_{11}+x_{12}-(1+q)x_{21}-x_{22}.\\
	    	    \Psi^-(X)&:=\operatorname{tr}\left(\begin{bmatrix}1& -1-q^2\\1-q&-1
	    \end{bmatrix}\cdot \begin{bmatrix}x_{11}& x_{12}\\x_{21}&x_{22}
	    \end{bmatrix}\right)=x_{11}+(1-q)x_{12}-(1+q^2)x_{21}-x_{22}.
	\end{align*}
	For $\rho$ with an even number of parts we have:
	\begin{align}
	    \Psi^+(\clmat_q(\rho))=    \Psi^+(\clmat_q(0,\rho,0))=0, \label{eq:palin+}
	\end{align}
	For $\rho$ with an odd number of parts we have:
		\begin{align}
	    \Psi^-(\clmat_q(0,\rho))=    \Psi^+(\clmat_q(\rho,0))=0. \label{eq:palin-}
	\end{align}
	\end{lemma}
	\begin{proof}
	 The equalities may be proved by induction. Equation~\ref{eq:palin+} is quite straightforward as the matrix $\clmat_q(\rho)$ for any palindromic composition $\rho$ can be formed by starting with identity and multiplying by $U_0$ and $D_0$ from the sides. It is quite straightforward to verify via direct calculation that $\Psi^+(I)=0$ and for all $2\times2 $ matrices $X$ one has 
	 \begin{align*}
	 \Psi^+(U_0\cdot X\cdot D_0)=\Psi^+(D_0\cdot X\cdot U_0)=q\Psi^+(X)
	 .\end{align*} The situation of Equation~\ref{eq:palin-} is only slightly more complicated as the middle part of an odd palindromic composition can be odd or even. The result follows as  $\Psi^-(I)=\Psi^-(U_0)=\Psi^-(D_0)=0$ and for all $X$,
	 \begin{align*}
	 \Psi^-(U_0\cdot X\cdot U_0)=\Psi^-(D_0\cdot X\cdot D_0)=q\Psi^-(X).
	 \end{align*}  
	\end{proof}
	
	\begin{thm}\label{thm:3} Let $\rho$ be a palindromic composition with an even number of parts. For $k\geq 1$, $r\geq1$ we have:
		\begin{align}
			\overline{\rank}((1,k,r+1,\rho,r);q)&= [k+1]_q \cdot \overline{\rank}((r+2,\rho,r);q), \tag{Id 1}\\
			\overline{\rank}((k,1,k+r,\rho,r);q)&=[k+1]_q \cdot
			\overline{\rank}((k+r+1,\rho,r);q)\tag{Id 2}.
		\end{align}
	\end{thm}
	\begin{proof} We will prove the statement by reducing both identities to $\Psi^+(\clmat_q(\rho))=0$ from Lemma~\ref{lem:palin}. Note that $(r,\rho,r)$ is palindromic if and only if $\rho$ is palindromic. So, if we denote the composition formed by adding $n$ to the \emph{leftmost} part of $\rho$ by $\rho+n$, it suffices to show the following hold for any palindromic composition $\rho$ with an even number of parts:
		\begin{align}
			\overline{\rank}((1,k,\rho+1);q)&= [k+1]_q \cdot \overline{\rank}((\rho+2);q) \tag{Id 1$'$},\\
			\overline{\rank}((k,1,\rho+k);q)&=[k+1]_q \cdot
			\overline{\rank}((\rho+k+1);q)\tag{Id 2$'$}.
		\end{align}
 By Equation~\ref{eq:circfen2}, (Id $1'$) can be expressed in terms of $\operatorname{tr}$ as:
\begin{align*}
   &\operatorname{tr}(U_0\cdot (D _0)^k\cdot U_0 \cdot \clmat_q(\rho) )=[k+1]_q\operatorname{tr}(U_0^2\cdot \clmat_q(\rho)),\\
    \Leftrightarrow &\operatorname{tr}(U_0\cdot ((D_0)^k-[k+1]_q)\cdot U_0 \cdot \clmat_q(\rho) )=0.	    
	\end{align*}
This expression is actually independent of $k$ as we have:
		\begin{align*}
(D_0)^k-[k+1]_q=\begin{bmatrix}
  0& -q[k]_q\\ [k]_q& -q[k-1]_q-[k+1]_q 
\end{bmatrix} =\begin{bmatrix}
    0 & -q[k]_q\\ [k]_q& -(1+q)[k]_q 
\end{bmatrix}=[k]_q \begin{bmatrix}
    0 & -q\\ 1 & -1-q
\end{bmatrix}.\end{align*}
This means $\clmat_q(\rho)$ satisfies (Id $1'$) if and only if:
\begin{align}
    \operatorname{tr}\left( \begin{bmatrix}
        q&1 \\0 &1
    \end{bmatrix}\begin{bmatrix}
    0 & -q\\ 1 & -1-q
\end{bmatrix}\begin{bmatrix}
        q&1 \\0 &1
    \end{bmatrix}\clmat_q(\rho)\right)=   q \Psi^+(\clmat_q(\rho))=0.\label{eq:midpalind}
\end{align}
Similarly, we can rewrite (Id  2$'$) in terms of trace of rank matrices: 
\begin{align*}
    &\operatorname{tr}(U_0^k\cdot (D _0) \cdot (U_0)^k \cdot \clmat_q(\rho) )=[k+1]_q\operatorname{tr}(U_0^{k+1}\cdot \clmat_q(\rho)) \\
   \Leftrightarrow  & \operatorname{tr}(U_0\cdot (U_0^{k-1}\cdot D_0-[k+1]_q) \cdot (U_0)^{(k-1)}\cdot U_0 \cdot \clmat_q(\rho))=0.
\end{align*}

Again, this expression ends up not being dependent on the value of $k$.
	\begin{align*}
	   (U_0^{k-1}\cdot D_0-[k+1]_q) \cdot (U_0)^{k-1})=q^{k-1}\begin{bmatrix}
	       0 & -q \\ 1 & -1 -q
	   \end{bmatrix},
	\end{align*}
	which means by Equation~\ref{eq:midpalind}, the condition to satisfy (Id 2$'$) is also that $\Psi^+(\clmat_q(\rho))=0$.
		\end{proof}

\begin{soru} Equation~\ref{eq:palin+} for palindromic compositions of even length gives rise to the rank polynomial identities listed above. Are there similar consequences to Equation~\ref{eq:palin-}? Can we find rank polynomial identities for palindromic compositions of odd length?
\end{soru}			

	\begin{thm} For $k\geq 2$, $s\geq 0$ we have:
		\begin{align*}
		\overline{\rank}((2k,k^{2s+1});q)&=(1+q^k)\cdot \overline{\rank}((k+1,k-1,k^{2s});q).\tag{Id 3}\\
		\end{align*}
	\end{thm} 
	\begin{proof} Set $A:=U_0^k D_0^k$. By Equation~\ref{eq:circfen2}, we can rewrite (Id 3) in terms of traces:
	\begin{align*}
	   & \operatorname{tr}(U_0^{2k} D_0^k  A^s)=(1+q^k) \operatorname{tr}(U_0^{k+1} D_0^{k-1} A^s)\\
	   \Leftrightarrow & \operatorname{tr}((U_0^{2k} D_0^k-(1+q^k)U_0^{k+1} D_0^{k-1})\cdot A^s)=0.
	\end{align*}
It is in fact enough to prove the identity for $A^0$ and $A^1$ as the characteristic polynomial of $A$ has degree $2$. At $s=0$ we have:
\begin{align*}
    U_0^{2k} D_0^k-(1+q^k)U_0^{k+1} D_0^{k-1}=\displaystyle \frac{q^{k-1}}{q-1}\begin{bmatrix}
        -1+q+q^2-q^k&1-q^2-q^3+q^k+q^{k+2}-q^{k+3}\\
1+2q-q^k &1-q-q^2+q^k.
    \end{bmatrix}
\end{align*} The trace of the above matrix is $0$. One can do a similar calculation to show the trace is zero for the $s=1$ case as well.
	\end{proof}
	
\begin{soru} It might be interesting to see if (Id 3) above generalizes to an identity on palindromic compositions.
\end{soru}

	\section{A Combinatorial Model for $q$-deformed Markov Numbers}\label{sec:cohn}
 In the introduction, we have mentioned that all positive integer triples solving the Markov Diophantine equation $x^2+y^2+z^2=3xyz$ can be reached from the inital solution $(1,1,1)$ by operations of the type $(a,b,c)\rightarrow(a,b,(a^2+b^2)/c)$. Equivalently, one can make use of Christoffel words. Starting with the triple $(a,ab,b)$, all the triples we can obtain by the operations:
	\begin{align*}
		& (u,uv,v) \longrightarrow (u,uuv,uv)\\
		& (u,uv,v) \longrightarrow (uv,uvv,v)
	\end{align*} are called \definition{Christoffel triples}, and the words $w(a,b)$ that occur as coordinates in these triples are called \definition{Christoffel words}. Let $W_C(a,b)$ denote the set of Christoffel words. Given a Christoffel word $w(a,b)$, the corresponding Markov number can be obtained by calculating the matrix $w(A,B)$, and then dividing the trace by three, where $A$ and $B$ are the \definition{Cohn matrices} defined as:
 	\begin{align*}\displaystyle
		&A:=\begin{bmatrix}
			2& 1\\ 1 & 1
		\end{bmatrix} \qquad \qquad B:=\begin{bmatrix}
			5 &2\\
			2 &1\\
		\end{bmatrix} 
	\end{align*}

	In \cite{leclere}, the following $q$-transformation is given for the Cohn matrices:
	\begin{align*}\displaystyle
		&[A]_q:=\begin{bmatrix}
			q+q^2& 1\\ q & 1
		\end{bmatrix} \qquad \qquad [B]_q:=\begin{bmatrix}
			q+2q^2+q^3+q^4 &1+q\\
			q+q^2 &1\\
		\end{bmatrix} 
	\end{align*}
	
	A similar deformation involving negative powers of $q$ is given in \cite{kogiso} where  $[A]_q$ is replaced with $q^{-1}[A]_q$ and  $[B]_q$ is replaced with $q^{-2}[B]_q$. The two deformations are equivalent up to a multiplication with a power of $q$, we are following the \cite{leclere} version that takes $q^0$ as the smallest term for ease of notation.
 
 For a given $w \in W_C(a,b)$, we will denote the length of $w$ by $\ell(w)=\ell_a(w)+\ell_b(w)$ where $\ell_a(w)$ (resp. $\ell_b(w)$) stands for the number of times $a$ (resp. $b$) occurs in $w$. If $w(a,b)\in W_C$ is not equal to the trivial words $a$ and $b$, then we also define $\hat{w}(a,b)$ to be the palindromic word satisfying $w(a,b)=a\,\hat{w}(a,b)\,b$. Finally, we set $w_q(A,B):=w({[A]_q},{[B]_q})$ for ease of notation.
	
	\begin{thm}[\cite{kogiso}, Theorem 2.6] For a Christoffel triple $(w,w\,w',w')$, $$(x,y,z):=([3]_q)^{-1}\cdot(\mathrm{tr}(w_q(A,B),\mathrm{tr}(w_q(A,B)w'_q(A,B)),\mathrm{tr}(w'_q(A,B))) $$ satisfies the following $q$-deformed Markov equation :
		\begin{equation}
			\left(\frac{x}{q^{c_x}}\right)^2+\left(\frac{y}{q^{c_y}}\right)^2+\left(\frac{z}{q^{c_z}}\right)^2 + \frac{(q-1)^3}{q^3} =[3]_q \frac{xyz}{q^{c_x+c_y+c_z}}.
		\end{equation}
		where $c_x:=\ell(w)+\ell_b(w)$, $c_z:=\ell(w')+\ell_b(w')$ and $c_y:=c_x+c_z$.
	\end{thm}
	
	We will now give a combinatorial description for entries of the Cohn matrices using oriented posets. Consider the following posets:
	
	
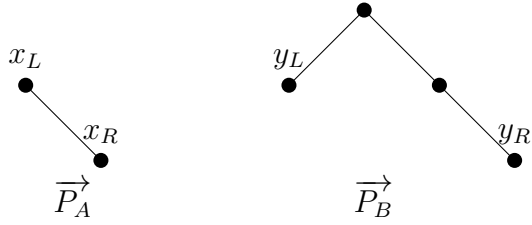
\begin{center}
		\begin{tikzpicture}[scale=.5]
			\node [] (0) at (-2, 2) {};
			\fill(-2,2) circle(.2);
			\node  (1) at (0, 0) {};
			\node  (2) at (5, 2) {};
			\node (3) at (7, 4) {};
			\node (4) at (9, 2) {};
			\node  (5) at (11, 0) {};
			\fill(0,0) circle(.2);
			\fill(5,2) circle(.2);
			\fill(7,4) circle(.2);
			\fill(9,2) circle(.2);
			\fill(11,0) circle(.2);
			\node at (-2, 2.7) {$x_L$};
			\node at (0, 0.7) {$x_R$};
			\node  at (5, 2.7) {$y_L$};
			\node  at (11, 0.7) {$y_R$};
			\node  at (-0.75, -1) {$\overrightarrow{P_A}$};
			\node  at (7.25, -1) {$\overrightarrow{P_B}$};
			\draw (1.center) to (0.center);
			\draw (2.center) to (3.center);
			\draw (3.center) to (5.center);
		\end{tikzpicture}
	\end{center}
	
	These deformations will help us view any Cohn matrix as a rank matrix of a fence poset, whose trace is given by a circular rank polynomial. For a given word $w(A,B)$, form a corresponding composition $\alpha(w)$ by replacing each $B$ with a pair of parts of size $2$ and each $A$ with a pair of parts of size $1$. For example, if $w(A,B)=AAAB$ we get $\alpha(w)=(1,1,1,1,1,1,2,2).$ 
	
	\begin{prop} The $q$-deformations $[A]_q$ and $[B]_q$ correspond to the dual rank matrices for the oriented posets $\overrightarrow{P_A}$ and $\overrightarrow{P_B}$ respectively. As a consequence, for any word $w(A,B)$ we have
		\begin{equation}
			\mathrm{tr}(w_q(A,B))=\overline{\rank}(\alpha(w)) \label{eq=cohntotrace}
		\end{equation}
	\end{prop}
	
	\begin{proof} Calculating the corresponding rank polynomials shows us that $\drm_q(\overrightarrow{P_A})=[A]_q$ and $\drm_q(\overrightarrow{P_B})=[B]_q$. The matrix product $[A]_q^k\,[B]_q$ therefore corresponds to the poset formed by connecting copies of $A$ and $B$. With the connected edges included, any copy of $B$ corresponds to $2,2$ and any copy of $A$ corresponds to $1,1$ . As seen in the example  $\overrightarrow{P_A}\revdiplus\overrightarrow{P_A}\revdiplus\overrightarrow{P_A}\revdiplus\overrightarrow{P_B}$ below, we end up with the fence poset for $\alpha(w)$ with the leftmost edge missing.  Taking the trace is equivalent to forming the missing edge by connecting the two ends, giving us $\overline{\rank}(\alpha(w))$.
		
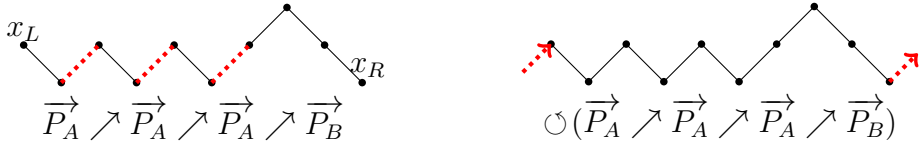
		\begin{figure}[ht]
  \centering
			\begin{tikzpicture}[scale=.25]
				\fill(2,2) circle(.2);
				\fill(4,0) circle(.2);
				\fill(6,2) circle(.2);
				\fill(8,0) circle(.2);  
				\fill(10,2) circle(.2);
				\fill(12,0) circle(.2);
				\fill(14,2) circle(.2);
				\fill(16,4) circle(.2); 
				\fill(18,2) circle(.2); 
				\fill(20,0) circle(.2); 
				\draw(2,2)--(4,0) (6,2)--(8,0) (10,2)--(12,0) (14,2)--(16,4)--(20,0);
				\draw[ultra thick,red,dotted] (4,0)--(6,2) (8,0)--(10,2) (12,0)--(14,2);
				\node at (2, 2.7) {$x_L$};
				\node at (20.3, 0.8) {$x_R$};
				\node at (11,-2) {$\overrightarrow{P_A}\revdiplus\overrightarrow{P_A}\revdiplus\overrightarrow{P_A}\revdiplus\overrightarrow{P_B}$};
			\end{tikzpicture}\qquad \qquad\begin{tikzpicture}[scale=.25]
				\fill(2,2) circle(.2);
				\fill(4,0) circle(.2);
				\fill(6,2) circle(.2);
				\fill(8,0) circle(.2);  
				\fill(10,2) circle(.2);
				\fill(12,0) circle(.2);
				\fill(14,2) circle(.2);
				\fill(16,4) circle(.2); 
				\fill(18,2) circle(.2); 
				\fill(20,0) circle(.2); 
				\draw(2,2)--(4,0) (6,2)--(8,0) (10,2)--(12,0) (14,2)--(16,4)--(20,0);
				\draw[] (4,0)--(6,2) (8,0)--(10,2) (12,0)--(14,2);
				\draw[ultra thick, red,dotted,->] (20,0)--(21.5,1.5);
				\draw[ultra thick, red,dotted,->] (0.5,.5)--(2,2);
				\node at (11,-2) {$\revclose(\overrightarrow{P_A}\revdiplus\overrightarrow{P_A}\revdiplus\overrightarrow{P_A}\revdiplus\overrightarrow{P_B})$};
			\end{tikzpicture}
			
		\end{figure}

	\end{proof}
	
	\begin{thm}\label{thm:markov} For a  Christoffel $ab$-word $w(a,b)=a\hat{w}(a,b)b$ we have:
		\begin{equation}\label{eq:christoffel}
			\mathrm{tr}(w_q(A,B))=\overline{\rank}((1,1,\alpha(\hat{w}),2,2);q)=[3]_q\cdot\overline{\rank}((3,1,\alpha(\hat{w});q).
		\end{equation}
		In particular,  \begin{equation*}  \mathrm{tr}(w_q(A,B))\in [3]_q \mathbb{N}[q].\end{equation*}
  Consequentally, the $q$-deformed Markov number corresponding to $w(a,b)=a\hat{w}(a,b)b$ is given by the rank polynomial of the circular fence poset $\overline{F}(3,1,\alpha(\hat{w}))$. It is symmetric and unimodal.
	\end{thm}
 
	\begin{proof} As the word $\hat{w}$ is palindromic and consists of pairs of parts, the corresponding permutation $\alpha(\hat{w})$ is palindromic with an even number of parts. Applying (Id 1) with $k=2$ and $r=1$ gives us:
	\begin{align*}
	    \overline{\rank}((1,2,2,\alpha(\hat{w}),1);q)=[3]_q \overline{\rank}((3,\alpha(\hat{w}),1);q).
	\end{align*} This is equivalent to Equation~\ref{eq:christoffel}  as moving matrices cyclically does not change trace and we can reverse the order by the symmetry of circular rank polynomials. The symmetry and unimodality are shown in \cite{ourpaper} and \cite{chainlink} respectively.
	\end{proof}

 		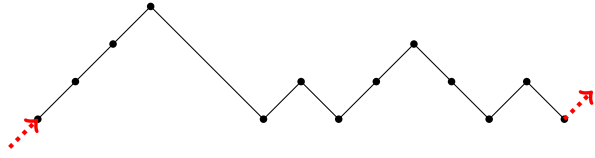
\begin{figure}[ht]
  \centering
			\begin{tikzpicture}[scale=.25]
   \fill(-4,4) circle(.2);
    \fill(-6,2) circle(.2);
       \fill(-8,0) circle(.2);
				\fill(-2,6) circle(.2);
				\fill(4,0) circle(.2);
				\fill(6,2) circle(.2);
				\fill(8,0) circle(.2);  
				\fill(10,2) circle(.2);
				\fill(12,4) circle(.2);
				\fill(14,2) circle(.2);
				\fill(16,0) circle(.2); 
				\fill(18,2) circle(.2); 
				\fill(20,0) circle(.2); 
				\draw(-8,0)--(-2,6)--(4,0) (6,2)--(8,0) (10,2)--(12,4) (14,2)--(16,0)--(18,2)--(20,0);
				\draw[] (4,0)--(6,2) (8,0)--(10,2) (12,4)--(14,2);
				\draw[ultra thick, red,dotted,->] (20,0)--(21.5,1.5);
				\draw[ultra thick, red,dotted,->] (-9.5,-1.5)--(-8,0);
			\end{tikzpicture}
\caption{The circular fence poset $\overline{F}(3,1,1,1,2,2,1,1)$}\label{fig:markovexample}
		\end{figure}
  For example, for the Christoffel word $a^2bab=a(aba)b$, we get the circular fence poset for $(3,1,1,1,2,2,1,1)$ shown in Figure~\ref{fig:markovexample}. The corresponding $q$-deformed Markov number is given by:
  \begin{equation*}
\overline{\rank}((3,1,1,1,2,2,1,1);q)=1+4q+9q^2+16q^3+23q^4+29q^5+30q^6+29q^7+23q^8+16q^9+9q^{10}+4q^{11}+q^{12}.
  \end{equation*}

See Table~\ref{tab:markovnumbers} for other examples of $q$-deformed numbers with corresponding Christoffel words and circular fence posets.
\begin{table}[]
    \centering
    \begin{tabular}{|c|c|l|}
    \hline
      $ab$   & $\overline{F}(3,1)$&$1+q+q^2+q^3+q^4$ \\
 $a^2b$   & $\overline{F}(3,1,1,1)$&$1+2q+2q^2+3q^3+2q^4+2q^5+q^6$ \\
 $ab^2$   & $\overline{F}(3,1,2,2)$&$1+2q+4q^2+5q^3+5q^4+5q^5+4q^6+2q^7+q^8$ \\
 $a^3b$   & $\overline{F}(3,1,1,1,1,1)$&$1+3q+4q^2+6q^3+6q^4+6q^5+4q^6+3q^7+q^8$ \\
  $abab^2$   & $\overline{F}(3,1,2,2,1,1,2,2)$&$1+4q+11q^2+22q^3+36q^4+50q^5+60q^6+65q^7+60q^8+\cdots+q^{14}$ \\
 $ab^2$   & $\overline{F}(3,1,2,2,2,2)$&$1+3q+8q^2+14q^3+20q^4+25q^5+27q^6+25q^7+20q^8+\cdots+q^{12}$ \\
  $a^4b$   & $\overline{F}(3,1,1,1,1,1,1,1)$&$1+4q+7q^2+11q^3+14q^4+15q^5+14q^6+11q^7+7q^8+4q^9+q^{10}$ \\
    $ab^4$   & $\overline{F}(3,1,2,2,2,2,2,2)$&$1+4q+13q^2+29q^3+53q^4+82q^5+110q^6+131q^7+139q^8+\cdots+q^{16}$ \\
 \hline
    \end{tabular}
    \caption{Christoffel Words with corresponding fences and Markov polynomials.}
    \label{tab:markovnumbers}
    \end{table}

\section{Generalized Oriented Posets}
\label{sec:further} In this section, we will consider a generalization of oriented posets where instead of one specialized vertex for each end, we pick a list of left end vertices $L=(L_1,L_2,\ldots, L_t)$
	and a list of right end vertices $(R_1,R_2,\ldots,R_s)$ on the poset $P$. We call the structure $\overrightarrow{P}=(P, L,R)$ we obtain a \definition{ generalized (t,s)-oriented poset}. 

	As in Section~\ref{sec:oriented posets}, we consider subsets of the distributive lattice $J(P)$. We use a subscript on the right to describe which nodes on $R$ should be included, and a subscript on the left to describe which nodes on $L$ should be \emph{excluded}. Let $A$ be a subset of $[t]$ and $B$ be a subset of $[s]$. We set: 
	\begin{align*}
		\rank(\overrightarrow{P};q):=& \sum_{I\in J(P)}  q^{|I|}  \qquad \qquad  &_{{A}}\rank(\overrightarrow{P};q):=\sum_{\substack{I\in J(P)\\ i\in A \Rightarrow L_i\notin I} }  q^{|I|} \\
		R_B(\overrightarrow{P};q):=& \sum_{\substack{I\in J(P)\\ i\in B \Rightarrow \rank_i\in I}} q^{|I|}  \qquad &
		_A R_B(\overrightarrow{P};q):= \sum_{\substack{I\in J(P)\\  i\in A \Rightarrow L_i\notin I\\ i\in B \Rightarrow R_i\in I} }  q^{|I|}\\
		\rank_{\dualfix{B}}(\overrightarrow{P};q):=& \sum_{\substack{I\in J(P)\\ i\in {{B}} {\color{red}{\Leftrightarrow}} R_i\in I}} q^{|I|}  \qquad &
		_A \rank_{\dualfix{B}}(\overrightarrow{P};q):= \sum_{\substack{I\in J(P)\\  i\in A \Rightarrow L_i\notin I\\ i\in {{B}} {\color{red}{\Leftrightarrow}} \rank_i\in I}} q^{|I|} 
	\end{align*}
	Generalized $(t,s)$-oriented posets have rank matrices of size $2^t \times 2^s$, where rows are indexed by subsets of $[t]$ and columns are indexed by subsets of $s$. When $t=s=4$  we get the following matrices:
	\begin{align*}
	\rmm_q(\overrightarrow{P})&=\begin{bmatrix}
	    R & - R_{\{1\}}& - R_{\{2\}} &  R_{\{1,2\}}\\
	      _{\{1\}}R & - _{\{1\}}R_{\{1\}} & - _{\{1\}}R_{\{2\}} &  _{1}R_{\{1,2\}}\\
	      _{\{2\}}R & - _{\{2\}}R_{\{1\}} & - _{\{2\}}R_{\{2\}} &  _{\{2\}}R_{\{1,2\}}\\
	      _{\{1,2\}}R & - _{\{1,2\}}R_{\{1\}} & - _{\{1,2\}}R_{\{2\}} &  _{\{1,2\}}R_{\{1,2\}}
	\end{bmatrix},\\	\drm_q(\overrightarrow{P})&=
	\begin{bmatrix}
	    R_{\dualfix{\{1,2\}}} & R_{\dualfix{\{2\}}}&  R_{\dualfix{\{1\}}} &  R_{\dualfix{\varnothing}}\\
	      _{\{1\}}R_{\dualfix{\{1,2\}}} &  _{\{1\}}R_{\dualfix{\{2\}}} &  _{\{1\}}R_{\dualfix{\{1\}}} &  _{1}R_{\dualfix{\varnothing}}\\
	      _{\{2\}}R_{\dualfix{\{1,2\}}} & _{\{2\}}R_{\dualfix{\{2\}}} &  _{\{2\}}R_{\dualfix{\{1\}}} &  {_{\{2\}}R_{\dualfix{\varnothing}}}\\
	      _{\{1,2\}}R_{\dualfix{\{1,2\}}} &  _{\{1,2\}}R_{\dualfix{\{2\}}} & { _{\{1,2\}}R_{\dualfix{\{1\}}} }&  _{\{1,2\}}R_{\dualfix{\varnothing}}.
	\end{bmatrix}
	\end{align*}
	More generally, the rank matrix and the dual rank matrix of a generalized (t,s)-oriented poset $\overrightarrow{P}$ are matrices of size $2^{|L|}\times 2^{|R|} $ defined as follows:
	\begin{align*}
		\rmm_q(\overrightarrow{P}):=((-1)^{|B|} \, _A \rank_B)_{A,B},\qquad 
		\drm_q(\overrightarrow{P}):=\,( _A \rank_{\dualfix{B^c}})_{A,B}.
	\end{align*}
	Here $A$ and $B$ range over subsets of $[t]$ and $[s]$ respectively:  $\varnothing \subseteq A\subsetneq [t]$, 	$\varnothing\subsetneq B\subseteq [s]$.	Note that when $t=s=1$, we recover the rank matrices we previously defined. 
	
	For a generalized $(t,s)$-oriented poset $\overrightarrow{P}=(P,L_P,R_P)$ and a generalized $(s,r)$-oriented poset $\overrightarrow{Q}=(Q,L_Q,R_Q)$, we set $\overrightarrow{P}\diplus\overrightarrow{Q}=(S,L_P,R_Q)$ where $S$ is the poset given by connecting $P$ and $Q$ by adding relations ${R_P}_i\succeq {L_P}_i $ for all $i \in S$. Also, when $t=s$ we denote the structure obtained by adding the relations ${R}_i\succeq {L}_i $ for all $i \in S$ by $\close(\overrightarrow{P})$.
	\begin{prop} The rank matrices for generalized oriented posets satisfy the following:
		\begin{align*}
			\rmm_q(\overrightarrow{P}\diplus \overrightarrow{Q})&=\rmm_q(\overrightarrow{P})\cdot\rmm_q( \overrightarrow{Q}),
			\qquad \qquad \rank(\close(\overrightarrow{P});q)=\mathrm{tr}(\rmm_q(\overrightarrow{P})),\\
			\rmm_q(\overrightarrow{P}\revdiplus \overrightarrow{Q})&=\drm_q(\overrightarrow{P})\cdot\rmm_q( \overrightarrow{Q}),\\
			\drm_q(\overrightarrow{P}\revdiplus \overrightarrow{Q})&=\drm_q(\overrightarrow{P})\cdot\drm_q( \overrightarrow{Q}), \qquad \qquad
			\rank(\revclose(\overrightarrow{P});q)=\mathrm{tr}(\drm_q(\overrightarrow{P})),\\
					\drm_q(\overrightarrow{P}\diplus \overrightarrow{Q})&=\rmm_q(\overrightarrow{P})\cdot\drm_q( \overrightarrow{Q}).
		\end{align*}
		
	\end{prop}

The proofs are similar to the case with $t=1$ and $s=1$. The multiplication of matrices corresponds to adding the connecting relations, by inclusion exclusion for $\rmm_q$ case, and adding up disjoint possiblities for the $\drm_q$ case. The conditions on $L_P$ and $R_Q$ remain unchanged, so those sets are carried over as left and right connecting node sets of the product. The trace works similarly. One can go from one matrix to the other by right multiplication with an invertible matrix as in the $t=s=1$ case, so that multiplying regular and dual rank matrices behaves nicely.

The choice of inclusion-exclusion for the regular rank matrix and taking disjoint possiblities for the dual addition is not dictated by what we want to achieve, but a particular choice made in order to match the existing notation. 

As the size of the matrix increases exponentially with the size of $t$ and $s$, the setting of generalized oriented matrices might not be very practical for studying posets. There are cases where it comes in very handy however. One case is to describe end-points of the poset where no more addition is planned. In that case one can take $s$ or $t$ to be $0$, and simplify calculations. If this is done at both left and right ends, one ends up with only the rank polynomial as the output. There is also the option of simplifying the matrices when the selected left and right end-points have relations among them, making some of the rows redundant. The simplest case is when one adds the condition that the nodes on $L$ and $R$ are chains. In that case, one can use $(s+1) \times (t+1)$ matrices to convey the information needed for adding up posets.

\section{Comments and Further Directions}

\begin{itemize}
    \item \textbf{Rank Polynomials of Self-Dual Partitions:}
    The Ferrers diagram of a partition $\lambda$ has a natural poset structure given by replacing the boxes with nodes of the poset, as seen in Figure~\ref{fig:further} below. The ideals of the poset correspond to nodes below $\lambda$ in the Hasse diagram of the Young Lattice, so the rank polynomial can be seen as the generating function of the size statistic on an interval in the Young lattice. It was conjectured in 1990 (see \cite{stanton}) that if $\lambda$ is a self-dual partition (it is equal to its transpose) then the corresponding rank polynomial is unimodal.
    
    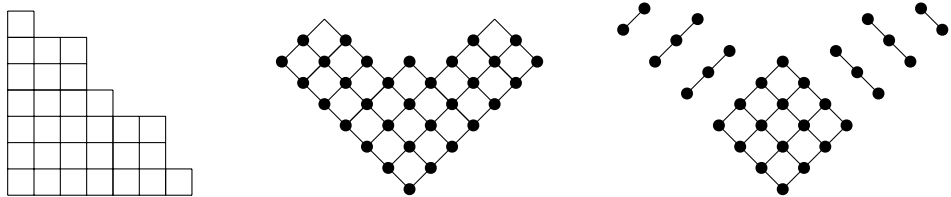
\begin{figure}[ht]
        \centering
 \begin{tikzpicture} [scale=.35]
\draw (0,0) grid (4,4);
   \draw (0,4) grid (3,6);
   \draw (4,0) grid (6,3);
   \draw (0,6) grid (1,7);
   \draw (6,0) grid (7,1);
 \end{tikzpicture} $\qquad$
 \begin{tikzpicture} [scale=.4]
 \begin{scope} [rotate=45] \draw (0,0) grid (3,3);
   \draw (0,3) grid (2,5);
   \draw (3,0) grid (5,2);
   \draw (0,5) grid (2,6);
   \draw (5,0) grid (6,2);
   \foreach \x in {0,...,3}
    \foreach \y in {0,...,3} {\fill(\x,\y) circle(.2); }
    \foreach \x in {0,1,2}
    \foreach \y in {4,5} {\fill(\x,\y) circle(.2);  \fill(\y,\x) circle(.2); }
    \foreach \x in {0,1}
    \foreach \y in {6} {\fill(\x,\y) circle(.2);  \fill(\y,\x) circle(.2); }
    \end{scope}
 \end{tikzpicture}$\qquad$
 \begin{tikzpicture} [scale=.4]
 \begin{scope} [rotate=45] \draw (0,0) grid (3,3);
   \draw (0,4.5)--(2,4.5) (0,6)--(2,6) (0,7.5)--(1,7.5) (4.5,0)--(4.5,2) (6,2)--(6,0) (7.5,0)--(7.5,1);
   \foreach \x in {0,...,3}
    \foreach \y in {0,...,3} {\fill(\x,\y) circle(.2); }
    \foreach \x in {0,1,2}
    \foreach \y in {4.5,6} {\fill(\x,\y) circle(.2);  \fill(\y,\x) circle(.2); }
    \foreach \x in {0,1}
    \foreach \y in {7.5} {\fill(\x,\y) circle(.2);  \fill(\y,\x) circle(.2); }
    \end{scope}
 \end{tikzpicture}
        \caption{The Young diagram for $(7,6,4,3,3,1)$ (left) with corresponding poset (middle) and decomposition (right).}
        \label{fig:further}
    \end{figure}

One can use the generalized rank matrices defined in Section~\ref{sec:further} to decompose the poset for $\lambda$ into one central rectangular lattice corresponding to the Durfee square of $\lambda$ multiplied with chains of decreasing sizes from both sides. If $\lambda$ is self-dual, as seen in Figure~\ref{fig:further}, the added chains will be paired up, as will the corresponding matrices. It might be interesting to examine whether the symmetry of the structure can be used to make an argument towards unimodality.

    \item 	\textbf{Resolution of Conjecture~\ref{ourconj}:} After this paper was shared on arXiv, the author, together with Ravichandran and Özel, was able to resolve Conjecture~\ref{ourconj} about the unimodality of circular rank polynomials using characteristic polynomials of rank matrices. This is one of the results in the upcoming paper \cite{chainlink}. This also proves Conjecture~\ref{con:lec+}, fully characterizing  the cases where $\mathrm{tr}(M_q(c_1,c_2,\ldots,c_k))$ for $c_i\geq 2$ is not 
    unimodal.
    
    \item \textbf{Generating functions and Cluster Algebras:} The rank polynomial keeps track of the number of elements of each ideal of a poset, but the machinery described here would work equally as well with any other statistic. In particular, for a poset on nodes $x_1,x_2,\cdots,x_k$, one can replace the rank polynomial by the generating polynomial $\sum_I \prod_{i\in I} x_i$ at each step to get the generating function of the ideals of the matrix.
    
    This particular case has an application in cluster algebras, namely, it allows us to calculate the expansion formula of a curve in terms of a chosen basis using matrices. In the recent work \cite{cluster}, the author and Emine Yıldırım were able to use oriented posets to obtain a matrix formulation for expansion formulas for curves in cluster algebras. This method is easily programmable and is able to give the results fairly quickly, especially in larger cases.

\textbf{Connection to recent work: } After an initial version of this work was shared on the arXiv, there were some developments that might lead to future connections. The recent paper on higher dimer covers has a larger matrix expression related to cluster algebras \cite{musiker2023higher}. We believe that an expansion of the weight matrices from \cite{cluster} to the generalized case that takes into account the particular nature of the underlying posets might provide a $q$-deformation for said matrix, as well as resolve some questions posed by the authors. Differential operators corresponding to $T_q$ and $S_q$ are used in \cite{thomas2023infinitesimal} to give a $q$-deformation of the Witt algebra. Leclere and Morier-Genoud also developed a new combinatorial model to study the numerators and denominators of $q$-deformed rationals using triangulated polygons and annuli. The connections between this model and the related cluster expansions might be interesting to explore.
    
\end{itemize}

\section*{Acknowledgements}
The author was supported by Tübitak BİDEB 2218 grant 121C285. The author would like to thank Valentin Ovsienko for pointing out the similarity between rank polynomials of circular fence posets and matrices from \ref{conj:leclere}. The author also would like to thank Mohan Ravichandran, Emine Yıldırım and Can Ozan Oğuz for helpful discussions on finalizing the notation for oriented posets and the anonymous referees for their thoughtful suggestions and pointing out related literature.
\typeout{}	
\bibliographystyle{elsarticle-num}
\bibliography{op}

\begin{thebibliography}{10}
\expandafter\ifx\csname url\endcsname\relax
  \def\url#1{\texttt{#1}}\fi
\expandafter\ifx\csname urlprefix\endcsname\relax\def\urlprefix{URL }\fi
\expandafter\ifx\csname href\endcsname\relax
  \def\href#1#2{#2} \def\path#1{#1}\fi

\bibitem{originalconj}
S.~Morier-Genoud, V.~Ovsienko,
  \href{https://doi.org/10.1017/fms.2020.9}{{$q$}-deformed rationals and
  {$q$}-continued fractions}, Forum Math. Sigma 8 (2020) Paper No. e13, 55.
\newblock \href {https://doi.org/10.1017/fms.2020.9}
  {\path{doi:10.1017/fms.2020.9}}.
\newline\urlprefix\url{https://doi.org/10.1017/fms.2020.9}

\bibitem{leclere}
L.~Leclere, S.~Morier-Genoud,
  \href{https://doi.org/10.1016/j.aam.2021.102223}{{$q$}-deformations in the
  modular group and of the real quadratic irrational numbers}, Adv. in Appl.
  Math. 130 (2021) Paper No. 102223, 28.
\newblock \href {https://doi.org/10.1016/j.aam.2021.102223}
  {\path{doi:10.1016/j.aam.2021.102223}}.
\newline\urlprefix\url{https://doi.org/10.1016/j.aam.2021.102223}

\bibitem{Saganpaper}
T.~{McConville}, B.~E. {Sagan}, C.~{Smyth}, {On a rank-unimodality conjecture
  of Morier-Genoud and Ovsienko}, {Discrete Math.} 344~(8) (2021) 13, id/No
  112483.
\newblock \href {https://doi.org/10.1016/j.disc.2021.112483}
  {\path{doi:10.1016/j.disc.2021.112483}}.

\bibitem{ourpaper}
E.~Kantarc\i~O\u{g}uz, M.~Ravichandran,
  \href{https://doi.org/10.1016/j.disc.2022.113218}{Rank polynomials of fence
  posets are unimodal}, Discrete Math. 346~(2) (2023) Paper No. 113218.
\newblock \href {https://doi.org/10.1016/j.disc.2022.113218}
  {\path{doi:10.1016/j.disc.2022.113218}}.
\newline\urlprefix\url{https://doi.org/10.1016/j.disc.2022.113218}

\bibitem{cluster}
E.~Kantarc\i~O\u{g}uz, E.~Yıldırım, Cluster algebras and rank matrices
  (2022).
\newblock \href {http://arxiv.org/abs/2211.08011} {\path{arXiv:2211.08011}}.

\bibitem{chainlink}
E.~Kantarcı~Oğuz, C.~Y. Özel, M.~Ravichandran, Chainlink polytopes and
  ehrhart-equivalence (2022).
\newblock \href {http://arxiv.org/abs/2211.08382} {\path{arXiv:2211.08382}}.

\bibitem{MR3908893}
T.~Yurikusa, \href{https://doi.org/10.1007/s10468-017-9755-3}{Cluster expansion
  formulas in type {A}}, Algebr. Represent. Theory 22~(1) (2019) 1--19.
\newblock \href {https://doi.org/10.1007/s10468-017-9755-3}
  {\path{doi:10.1007/s10468-017-9755-3}}.
\newline\urlprefix\url{https://doi.org/10.1007/s10468-017-9755-3}

\bibitem{MR4158076}
A.~Claussen, Expansion {P}osets for {P}olygon {C}luster {A}lgebras, ProQuest
  LLC, Ann Arbor, MI, 2020, thesis (Ph.D.)--Michigan State University.

\bibitem{thesis}
A.~Markoff, \href{https://doi.org/10.1007/BF01446234}{Sur les formes
  quadratiques binaires ind\'{e}finies}, Math. Ann. 17~(3) (1880) 379--399.
\newblock \href {https://doi.org/10.1007/BF01446234}
  {\path{doi:10.1007/BF01446234}}.
\newline\urlprefix\url{https://doi.org/10.1007/BF01446234}

\bibitem{markofftree}
E.~Bombieri, \href{https://doi.org/10.1016/j.exmath.2006.10.002}{Continued
  fractions and the {M}arkoff tree}, Expo. Math. 25~(3) (2007) 187--213.
\newblock \href {https://doi.org/10.1016/j.exmath.2006.10.002}
  {\path{doi:10.1016/j.exmath.2006.10.002}}.
\newline\urlprefix\url{https://doi.org/10.1016/j.exmath.2006.10.002}

\bibitem{chris}
C.~Reutenauer, \href{https://doi.org/10.1016/j.aam.2021.102179}{Christoffel
  words and weak {M}arkoff theory}, Adv. in Appl. Math. 127 (2021) Paper No.
  102179, 15.
\newblock \href {https://doi.org/10.1016/j.aam.2021.102179}
  {\path{doi:10.1016/j.aam.2021.102179}}.
\newline\urlprefix\url{https://doi.org/10.1016/j.aam.2021.102179}

\bibitem{MR4103773}
M.~Rabideau, R.~Schiffler,
  \href{https://doi.org/10.1016/j.aim.2020.107231}{Continued fractions and
  orderings on the {M}arkov numbers}, Adv. Math. 370 (2020) 107231, 18.
\newblock \href {https://doi.org/10.1016/j.aim.2020.107231}
  {\path{doi:10.1016/j.aim.2020.107231}}.
\newline\urlprefix\url{https://doi.org/10.1016/j.aim.2020.107231}

\bibitem{aignerbook}
M.~Aigner, \href{https://doi.org/10.1007/978-3-319-00888-2}{Markov's theorem
  and 100 years of the uniqueness conjecture}, Springer, Cham, 2013, a
  mathematical journey from irrational numbers to perfect matchings.
\newblock \href {https://doi.org/10.1007/978-3-319-00888-2}
  {\path{doi:10.1007/978-3-319-00888-2}}.
\newline\urlprefix\url{https://doi.org/10.1007/978-3-319-00888-2}

\bibitem{kogiso}
T.~{Kogiso}, q-deformations and t-deformations of {M}arkov triples (2020).

\bibitem{stanton}
D.~{Stanton}, {Unimodality and Young's lattice}, {J. Comb. Theory, Ser. A}
  54~(1) (1990) 41--53.
\newblock \href {https://doi.org/10.1016/0097-3165(90)90004-G}
  {\path{doi:10.1016/0097-3165(90)90004-G}}.

\bibitem{musiker2023higher}
G.~Musiker, N.~Ovenhouse, R.~Schiffler, S.~W. Zhang, Higher dimer covers on
  snake graphs (2023).
\newblock \href {http://arxiv.org/abs/2306.14389} {\path{arXiv:2306.14389}}.

\bibitem{thomas2023infinitesimal}
A.~Thomas, Infinitesimal modular group: $q$-deformed $\mathfrak{sl}_2$ and witt
  algebra (2023).
\newblock \href {http://arxiv.org/abs/2308.06158} {\path{arXiv:2308.06158}}.

\end{thebibliography}
\end{document}